\documentclass[10pt]{amsart}
\usepackage{mathtools}
\usepackage{kpfonts}

\newtheorem{theorem}{Theorem}[section]
\newtheorem{lemma}[theorem]{Lemma}
\newtheorem{trditev}[theorem]{Proposition}
\newtheorem{posledica}[theorem]{Corollary}

\theoremstyle{definition}

\newtheorem{example}[theorem]{Example}

\theoremstyle{remark}
\newtheorem{remark}[theorem]{Remark}

\numberwithin{equation}{section}

\usepackage{amssymb}
\usepackage{mathdots}
\usepackage{pst-node}
\usepackage{enumitem}
\usepackage{algpseudocode}
\usepackage{algorithmicx}

\def\codim{\mathop{\rm codim}\nolimits}

\def\Orb{\mathop{\rm Orb}\nolimits}

\parskip=\smallskipamount

\makeatletter 
\renewcommand\p@enumii{}
\makeatother

\makeatletter
\newcommand{\doublewidetilde}[1]{{%
  \mathpalette\double@widetilde{#1}%
}}
\newcommand{\double@widetilde}[2]{%
  \sbox\z@{$\m@th#1\widetilde{#2}$}%
  \ht\z@=.9\ht\z@
  \widetilde{\box\z@}%
}
\makeatother

\makeatletter
\renewcommand*\env@matrix[1][*\c@MaxMatrixCols c]{%
  \hskip -\arraycolsep
  \let\@ifnextchar\new@ifnextchar
  \array{#1}}
\makeatother

\begin{document}
\title[]
{
Isotropy groups of the action of orthogonal similarity on symmetric matrices 
}
\author{
Tadej Star\v{c}i\v{c}}
\address{Faculty of Education, University of Ljubljana, Kardeljeva Plo\v{s}\v{c}ad 16, 1000 Lju\-blja\-na, Slovenia}
\address{Institute of Mathematics, Physics and Mechanics, Jadranska
  19, 1000 Ljubljana, Slovenia}
\email{tadej.starcic@pef.uni-lj.si}
\date{August 13, 2021}


\keywords{isotropy groups, matrix equations, orthogonal matrices, symmetric matrices, unipotent group, Toeplitz matrices\\
\indent Research supported by grant P1-0291 
from ARRS, Republic of Slovenia.}

\begin{abstract} 
We find an algorithmic procedure that enables to compute and to describe the structure of the isotropy subgroups of the group of complex orthogonal matrices with respect to the action of similarity on complex symmetric matrices. 
A key step in our proof is to solve a certain rectangular block upper-triangular Toeplitz matrix equation. 
\end{abstract}

\maketitle

\vspace{-5mm}

\section{Introduction and the main result} 

All matrices considered in this paper are complex unless otherwise is stated. We use the notation $\mathbb{C}^{m\times n}$ for the set of matrices of size $m\times n$. 
By 
$S_n(\mathbb{C})$
we denote the vector space of all $n\times n$ symmetric matrices; $A$ is symmetric if and only if $A=A^{T}$. Let further $O_n(\mathbb{C})$ be the subgroup of orthogonal matrices in the group of nonsingular $n\times n$ matrices $GL_n(\mathbb{C})$. A matrix $Q$ is orthogonal if and only if $Q=(Q^{T})^{-1}$. 
The action of \emph{orthogonal similarity}
on $S_n(\mathbb{C})$ is defined as follows:
\begin{align}\label{aos}
\Phi\colon O_n(\mathbb{C})\times S_n(\mathbb{C})\to S_n(\mathbb{C}),\qquad (Q,A)\mapsto Q^{T}AQ.
\end{align}
%
The \emph{isotropy group} at $A\in S_n(\mathbb{C})$ with respect to the action (\ref{aos}) is 
%
\begin{equation}\label{isog}
\Sigma_A:=\{Q\in O_n(\mathbb{C})\mid Q^{T}AQ=A\}
\end{equation}
%
%
and the \emph{orbit} of $A$ is 
%
\begin{equation}\label{eqAQBs}
\Orb (A):=\{Q^{T}AQ \mid Q\in O_n(\mathbb{C})\}.  
\end{equation}
An orbit thus consists of orthogonally similar matrices and the isotropy groups of these matrices are isomorphic.

%
%
%

The action (\ref{aos}) 
describes symmetries of $S_n(\mathbb{C})$.
Hua's fundamental results \cite{Hua0,Hua1,Hua2} on the geometry of symmetric matrices assure that the study of symmetric matrices under $T$-congruence (which includes (\ref{aos})) is quite general.
An important information concerning a group action is provided by its orbits and the corresponding isotropy groups (see monographs \cite{GOV,Milne}), and to find these for the action (\ref{aos}) is the main purpose of this paper.
Moreover,  
the so-called linear isotropy representation at $A\in S_n(\mathbb{C})$ is the restricion of (\ref{aos}): 
\begin{equation}\label{tans}
\Sigma_A \times T_A \to T_A,\quad (Q,A)\mapsto Q^{T}AQ,\qquad T_A:= \{X^{T}A+AX\mid X=-X^{T}\in \mathbb{C}^{n\times n}\},
\end{equation}
%
a representation of $\Sigma_A$ on a 
complex vector space $T_A\subset S_n(\mathbb{C})$ associated to the tangent space of $\Orb(A)\subset S_n(\mathbb{C})$ at $A$ (see also Sec. \ref{PfTh}). 
It is closely related to invariant objects of $\Orb (A)$ (see \cite{GOV,KN}). 
%
On the other hand
(\ref{aos}) can be seen as a representation of $O_n(\mathbb{C})$; note that the classification of representations of complex classical groups along with their invariants is well understood (see e.g. \cite{Weyl}).
Finally, the isotropy groups of (\ref{aos}) are interesting from the linear algebraic point of view (check Remark \ref{OpI}).

To be able to compute the isotropy groups, it is essential to have simple representatives of orbits. 
Thus we recall the symmetric canonical form under similarity;
remember that symmetric matrices are similar if and only if they are orthogonally similar (see e.g. \cite{Gant}). 
Given a matrix $A$ with its Jordan canonical form:
\begin{equation}\label{JFs}
J(A)=\bigoplus_j J_{n_j}(\lambda_j), \qquad \lambda_j\in \mathbb{C}, 
\end{equation}
\vspace{-1mm}
where 
%
\begin{equation*}
     J_n(z):=\begin{bmatrix}
                                                      z    &  1       & \;     & 0    \\
						      \;     & z     & \ddots & \;    \\     
						      \;     & \;      & \ddots &  1     \\
                                                      0     & \;      & \;     & z   
                                   \end{bmatrix},\qquad z\in \mathbb{C}\qquad (n\textrm{-by-}n),
\end{equation*}
the symmetric canonical form is
\begin{equation}\label{NF1s}
S(A)=\bigoplus_{j}^{} K_{n_j}(\lambda_j),\qquad
\end{equation}
\vspace{-1mm}
in which
\begin{equation}\label{Sm}
K_n(z):=
\frac{1}{2}\left(
\begin{bsmallmatrix}
2z  & 1 &             & 0 \\
1 & \ddots &   \ddots        &   \\
   & \ddots  &  \ddots & 1 \\
0   &          & 1    & 2z \\
\end{bsmallmatrix}
+
i
\begin{bsmallmatrix}
0  &                &    -1  & 0\\
 &   \iddots           &   \iddots     & 1 \\
-1   &  \iddots  & \iddots &  \\
0   & 1          &     & 0 \\
\end{bsmallmatrix}
\right),\qquad z\in \mathbb{C}\qquad (n\textrm{-by-}n).
\end{equation}
It is uniquely determined up to a permutation of its direct summands. See \cite{Djok} for the tridiagonal symmetric canonical form.

Since the equation $Q^{T}AQ=A$ is equivalent to $(J(A))X=X(J(A))$ with $J(A)=PAP^{-1}$, $X=PQP^{-1}$, 
the following fact on isotropy groups follows immediately from the classical result on solutions of Sylvester's equation (see Theorem \ref{posls} (\ref{posls1})).

\vspace{-2mm}

\begin{trditev} \label{stabs11}
If $\lambda_1,\ldots,\lambda_k$ are distinct eigenvalues of  
$S=\bigoplus_{j=1}^{k}S_j$,
where each $S_j$ is a direct sum whoose summands are of the form (\ref{Sm}) and correspond to the eigenvalue $\lambda_j$,
it then follows that $\Sigma_{S}=\bigoplus_{j=1}^{k}\Sigma_{S_j}$. 
Furthermore, if $S_j=\lambda_j I_{n_j}$ for some index $j$, then $\Sigma_{S_j}=O_{n_j}(\mathbb{C})$. (We denote 
the $n\times n$ identity-matrix by $I_n$.)
\end{trditev}

\vspace{-1mm}

Therefore the isotropy groups under (\ref{aos}) of matrices with all distinct eigenvalues (hence with nonvanishing discriminants of their characteristic polynomials) are trivial. 
The situation in the generic case (on a complement of a complex analytic subset of codimension $1$) is thus quite simple.


Our aim is to inspect the nongeneric matrices (especially nondiagona\-li\-zable). The principal object of the investigation will be (up to similarity) the group of all nonsingular matrices commuting with a given square matrix $M$, i.e. nonsingular solutions of the homogeneous linear Sylvester's equation $MX=XM$;
see Sec. \ref{notation} for its properties.
First, recall that a \emph{block upper-triangular Toeplitz} matrix is: 
%
\small
\begin{align*}
&T(A_0,A_1,\ldots,A_{\beta-1})=\begin{bmatrix}
  A_{0} & A_{1}         & A_{2}             & \ldots   & \ldots &    A_{\beta-1}  \\
0       & A_0 & A_{1} & A_2  &      & \vdots \\
 \vdots & \ddots            & A_0               & A_1  & \ddots &   \vdots \\ 
 \vdots &  & \ddots   & \ddots             & \ddots &  A_2\\
 \vdots &        &              & \ddots  & \ddots &   A_1 \\
0       & \ldots            & \ldots & \ldots &  0      & A_0
\end{bmatrix}\qquad (\beta\textrm{-by-}\beta), 
\end{align*}
\normalsize
where $A_0,A_1,\ldots,A_{\beta-1}\in \mathbb{C}^{m\times n}$ and $T(A_0,A_1,\ldots,A_{\beta-1})=[T_{jk}]_{j,k=1}^{\beta}$ with $T_{jk}=0$ for $j>k$ and $T_{jk}=T_{(j+1)(k+1)}$.
Next, suppose $\alpha_1>\alpha_2>\cdots>\alpha_N$ and $m_1,\ldots,m_N\in \mathbb{N}$. Let $\mathcal{X}$ be an $N\times N$ block matrix such that its block $\mathcal{X}_{rs}$ 
is a rectangle $\alpha_r\times \alpha_s$ block upper-triangular Toeplitz matrix with blocks od size $m_r\times m_s$:
\begin{equation}\label{0T0}
\mathcal{X}=[\mathcal{X}_{rs}]_{r,s=1}^{N},\qquad
\mathcal{X}_{rs}=
\left\{
\begin{array}{ll}
[0\quad \mathcal{T}_{rs}], & \alpha_r<\alpha_s\\
\begin{bmatrix}
\mathcal{T}_{rs}\\
0
\end{bmatrix}, & \alpha_r>\alpha_s\\
\mathcal{T}_{rs},& \alpha_r=\alpha_s
\end{array}\right., \qquad b_{rs}=\min\{\alpha_s,\alpha_r\},
\end{equation}
in which 
$\mathcal{T}_{rs}$ is a $b_{rs}\times b_{rs}$ block upper-triangular Toeplitz matrix.
It turns out that orthogonal solutions of the equation $SX=XS$ with $S$ of the form (\ref{NF1s}) are related to matrices of the form (\ref{0T0}) such that the following properties are satisfied:
\begin{enumerate}[label={\bf (\Roman*)},ref={\Roman*},wide=0pt,itemsep=10pt]
\item \label{stabs0} The nonzero entries of $\mathcal{X}_{rs}$ for $r>s$ can be taken as free variables.

\item \label{stabs1} If $s=r$, then $\mathcal{X}_{rr}=T(A_0^{r},\ldots,A_{\alpha_r-1}^{r})$,
where $A_0^{r}\in O_{m_r}(\mathbb{C})$ can be any orthogonal matrix, and for $\alpha_r\geq 2$, $j\in \{1,\ldots,\alpha_r-1\}$ we have $A_{j}^{r}=A_0^{r}Z_j^{r}+D_j^{r}$ for some freely chosen skew-symmetric $Z_j^{r}=-(Z_j^{r})^{T}$ of size $m_r\times m_r$, 
and with $D_j^{r}\in \mathbb{C}^{m_r\times m_r}$ depending uniquely (and polynomially) on the entries of $A_0^{p}$, $Z_{\widetilde{j}}^{p}$ with $\widetilde{j}\in \{1,\ldots,j-1\}$, $p\in \{1,\ldots,N\}$ and on the entries of $\mathcal{X}_{pt}$ for $p,t\in \{1,\ldots,N\}$ with $p>t$. 
\item \label{stabs2} The entries of $\mathcal{X}_{rs}$ for $r<s$ are uniquely determined (the dependence is polynomial) by the entries of $\mathcal{X}_{pt}$ for $p,t\in \{1,\ldots,N\}$ with $p\geq t$.
\end{enumerate}

A simple example of a block diagonal matrix of the form (\ref{0T0}) is
%
\begin{align}\label{asZ}
&\mathcal{W}=\bigoplus_{r=1}^{N}T(I_{m_r},W_1^{r},\ldots,W_{\alpha_r-1}^{r}),\\
%
%
&W_{1}^{r}:=\frac{1}{2}Z_{1}^{r}, \qquad 
W_{n+1}^{r}:=\frac{1}{2}\big(Z_{n+1}^{r}-\sum_{j=1}^{n}(W_j^{r})^{T}W_{n-j+1}^{r}
\big), \quad n\geq 1,\nonumber
\end{align}
in which all $Z_{n}^{r}$ are skew-symmetric.
%
Another special matrix of the form (\ref{0T0}) contains the identity matrix as principal submatrix, formed by all blocks except those at the $p$-th and the $t$-th columns and rows, while 
blocks in the $p$-th and the $t$-th columns and rows are as follows: 
\vspace{-1mm}
\begin{align}\label{Hptk}
&\mathcal{G}_{p,t}^{k}(F)=[(\mathcal{G}_{p,t}^{k}(F))_{rs}]_{r,s=1}^{N},\qquad
(\mathcal{G}_{p,t}^{k}(F))_{rs}=
\left\{
\begin{array}{ll}
\begin{bsmallmatrix}
0\quad \mathcal{U}_{rs}
\end{bsmallmatrix}, & \alpha_r<\alpha_s\\
\begin{bsmallmatrix}
\mathcal{U}_{rs}\\
0
\end{bsmallmatrix}, & \alpha_r>\alpha_s\\
\mathcal{U}_{rs},
& \alpha_r=\alpha_s
\end{array}\right., \quad p< t,
\end{align}
where
\begin{align*}
&\mathcal{U}_{rs}=\left\{
\begin{array}{ll}
\oplus_{j=1}^{\alpha_r}I_{m_r}, &  r=s,\\
0,                        &  r\neq s 
\end{array}
\right., \quad \{r,s\}\not\subset\{p,t\},\qquad
\mathcal{U}_{rr}=T(I_{m_r},A_1^{rr},\ldots,A_{\alpha_r-1}^{rr}), \quad r\in \{p,t\},\nonumber\\ 
&
A_{j}^{pp}=\left\{\begin{array}{ll}
a_{n-1}(F^{T}F)^{n}, & j=n(2k+\alpha-\beta)\\
0,                      & \textrm{otherwise}
\end{array}
\right.,\qquad 
a_{n}=-\frac{1}{2^{2n+1}}\frac{1}{n+1}{2n \choose n},
\nonumber \\
&
A_{j}^{tt}=\left\{\begin{array}{ll}
a_{n-1}(FF^{T})^{n}, & j=n(2k+\alpha-\beta)\\
0,                      & \textrm{otherwise}
\end{array}
\right., \nonumber \\
&
\mathcal{U}_{pt}=
N_{\alpha_t}^{k}(F ),
\qquad
\mathcal{U}_{tp}=
N_{\alpha_t}^{k}(-F^{T}),
\qquad 0\leq k \leq \alpha_t-1,
\nonumber
\end{align*}
in which $N_{\beta}^{k}(F)$ is a $\beta\times \beta$ block matrix with $F\in \mathbb{C}^{m_p\times m_t}$ on the $k$-th diagonal above the main diagonal for $k\geq 1$ (on the main diagonal for $k=0$) and zeros othervise.

\begin{example}
$N=3$, $\alpha_1=4$, $\alpha_2=2$, $\alpha_3=1$, $m_1=2$, $m_2=3$, $m_3=1$; $F\in \mathbb{C}^{2\times 3}$:\\
\vspace{-1mm}
\small
\[
\mathcal{G}_{1,2}^{0}(F)=\begin{bmatrix}[cccc|cc|c]
I_{2} & 0 & -\tfrac{1}{2}F^{T}F & 0  &  -F^{T}  &  0  &  0\\
0   & I_{2} & 0   & -\tfrac{1}{2}F^{T}F  &  0  &   -F^{T}  &  0\\
0   & 0   & I_{2} & 0    &  0 & 0 & 0 \\
0   & 0   &  0  &  I_{2}   & 0 &0 & 0\\
\hline
0   & 0   & F & 0                         &  I_3  & 0  &  0\\
0   & 0   & 0   & F              &  0    &   I_3  &  0\\
\hline
0   & 0   & 0   & 0                        &  0    &   0    & 1 
\end{bmatrix}
\]
\normalsize
\end{example}

Our main result is the following.

\begin{theorem}\label{stabs}
%
%
If $S=\bigoplus_{r=1}^{N}\left(\bigoplus_{j=1}^{m_r} K_{\alpha_r}(\lambda)\right)$ for $\lambda\in \mathbb{C}$, 
then its isotropy group $\Sigma_{S}$ (with respect to (\ref{aos})) is isomorphic to the subgroup of the group of all invertible matrices of the form (\ref{0T0}) and such that its elements satisfy 
properties (\ref{stabs1}), (\ref{stabs1}), (\ref{stabs2}). 

Furthermore, $\Sigma_{\mathcal{S}}$ is isomorphic to a
semidirect product $\mathbb{O}\ltimes \mathbb{V}$, in which the subgroup $\mathbb{O}$ 
consists of all matrices of the form $\mathcal{Q}=\oplus_{r=1}^{N}(\oplus_{j=1}^{\alpha_r} Q_r)$ 
with $Q_r\in O_{m_j}(\mathbb{C})$, and a unipotent normal subgroup $\mathbb{V}$ (of order at most $\alpha_1-1$ and nilpotency class at most $\alpha_1$) generated by the set of matrices of the form (\ref{asZ}) and (\ref{Hptk}).
\end{theorem}

\vspace{-1mm}

We refer to \cite{Milne} for 
the theory of nilpotent and unipotent algebraic groups.  

\begin{remark}\label{OpI}
\begin{enumerate}[label=(\arabic*),ref={P\Roman*},wide=0pt,itemsep=2pt]
\item An algorithm to compute the isotropy groups is provided as (an essential) part of the proof of Theorem \ref{stabs}, more precisely, by Lemma \ref{EqT}. Due to technical reasons the lemma is stated and proved in Sec. \ref{cereq}. It describes the solutions of a certain rectangular block upper-triangular Toeplitz matrix equation, hence it
might be also of independent interest in matrix analysis. 
\item
To some extend Theorem \ref{stabs} could be applied to the problem of simultaneous reduction under $T$-congruence of a pair $(A,B)$ with $A$ arbitrary and $B$ nonsingular symmetric. We first make $B$ into the identity $I$ by applying Autonne-Takagi factorization and reduce $(A,B)$ to $(A',I)$. Next, we write $A'=C+Z$ with $S$ symmetric and $Z$ skew-symmetric. By a suitable orthogonal similarity transformation (keeping $I$ intact) we put $C$ into the symmetric normal form $S(C)$; we obtain $(S(C)+Z',I)$ with $Z'$ skew-symmetric. Finally, $Z'$ is simplified by
using the isotropy group of $S(C)$ with respect to (\ref{aos}) (keeping $I$, $S(C)$ intact). 
\end{enumerate}
\end{remark}

\vspace{-1mm}

The orbit $\Orb (A)$ of a matrix $A\in S_n(\mathbb{C})$ is an immersed complex submanifold in $S_n(\mathbb{C})$ and let $\codim (\Orb (A))$ be its codimension. Moreover, $\Orb (A)$ is biholomorphic to the quotient $O_n(\mathbb{C})/\Sigma_A$ (check e.g. \cite[Ch. II.1]{GOV}. Thus the following corollary is an immediate consequence of Theorem \ref{stabs}, although it can be easily proved by computing the tangent bundle of an orbit (see Sec. \ref{PfTh}).

\vspace{-1mm}

\begin{posledica}\label{orbdim}
%
If $\lambda_1,\ldots,\lambda_k$ are distinct eigenvalues of  
$S=
\bigoplus_{j=1}^{k}S_j$,
where each $\mathcal{S}_j$ is a direct sum whoose summands are of the form (\ref{Sm}) and correspond to the eigenvalue $\lambda_j$,
then $\codim (\Orb(S))=\sum_{j=1}^{k}\codim (\Orb(S_j))$.
Moreover, if $S=\bigoplus_{r=1}^{N}\left(\bigoplus_{j=1}^{m_r} S_{\alpha_r}(\lambda)\right)$ for $\lambda\in \mathbb{C}$, 
it then follows that
$\codim (\Orb(S))=\sum_{r=1}^{N}\alpha_r m_r\big(\tfrac{1}{2}(m_r+1)+\sum_{s=1}^{r-1}m_s\big)$.
\end{posledica}

\vspace{-1mm}

Note that the dimension of an orbit of $A$ in $\mathbb{C}^{n\times n}$ with respect to similarity is simply equal to the codimension of the set of solutions 
$AX=XA$ (see e.g. \cite[Section 30]{Arnold}), while 
in case of $T$-congruence 
one must solve $XA+AX^T=0$ (see \cite{TeranDopi2}).

\vspace{-1mm}


%
%
%
%
\section{Preliminaries}\label{notation}

\vspace{-1mm}

In this section we prepare some preliminary material.
%
First we recall a classical result on solutions of the Sylvester's equation; 
see e.g \cite[Chap. VIII]{Gant}.

\begin{theorem}\label{posls}
Let $J$ be of the form (\ref{JFs}). Suppose a matrix equation
\vspace{-1mm}
\begin{equation}\label{eqH1}
JX=XJ.
\end{equation}
%
\begin{enumerate}
\item \label{posls1} 
Assume that $J_{}^{}=\bigoplus_{r=1}^{N}J_{r}$,
in which all blocks of $J$ corresponding to the eigenvalue $\rho_r$ are collected together into $J^{}_{r}$. Then $X$ is a solution of the equation (\ref{eqH1}) if and only if it is of the form $X=\oplus_{r=1}^{N}X_{r}$ with $J^{}_{r}X_{r}=X_{r}J_{r}$. 
\item \label{posls2} Let $J_{}^{}=\bigoplus_{r=1}^{N}\bigoplus_{j=1}^{m_r} J_{\alpha_j}(\lambda)$ for $\lambda \in \mathbb{C}$ and $\alpha_{1}>\alpha_{2}>\ldots >\alpha_{N}$, and let $X$ be partitioned conformally to blocks as $J_{}$. Then $X$ is a solution of (\ref{eqH1}) if and only if  
$X=[X_{rs}]_{r,s=1}^{N}$ is such that every block $X_{rs}$ is further a $m_r\times m_s$ block matrix with blocks of size $\alpha_r\times \alpha_s$ and of the form 
\vspace{-1mm}
\small
\begin{equation}\label{QT}
\left\{\begin{array}{ll}
[0\quad T], & \alpha_r<\alpha_s \\
\begin{bmatrix}
T\\
0
\end{bmatrix}, & \alpha_r>\alpha_s\\
T, & \alpha_r=\alpha_s
\end{array}
\right.,
\end{equation}
\normalsize
\vspace{-1mm}
in which $T$ is an $b_{rs}$-by-$b_{rs}$ upper-triangular Toeplitz matrix ($b_{rs}=\min\{\alpha_r,\alpha_s\}$). 
%
\end{enumerate}
%
\end{theorem}

For our developments it is convenient to work with matrices with smaller number of blocks. This can be achieved by conjugating with a suitable permutation matrix (see e.g. \cite[Sec. 3.1]{Lin}).
Let $e_1,e_2,\ldots,e_{\alpha m}$ be the standard orthonormal basis in $\mathbb{C}^{\alpha m}$.
We set a permutation matrix formed by these vectors: 
\begin{equation}\label{perS}
\Omega_{\alpha,m}:=\left[e_1\;e_{\alpha+1}\;\ldots\;e_{(m-1)\alpha+1}\;e_2\;e_{\alpha+2}\;\ldots\;e_{(m-1)\alpha+2}\;\ldots\;e_{\alpha}\;e_{2\alpha}\;\ldots\;e_{\alpha m}\right].
\end{equation}
Observe that multiplication with $\Omega_{\alpha,m}$ from the right puts the $1$-st, the $(\alpha+1)$-th, \ldots, the $((m-1)\alpha+1)$-th column together,  further the $2$-nd, the $(\alpha+2)$-th, \ldots, the $((m-1)\alpha+2)$-th column together, and soforth. Similarly, multiplicating with $\Omega_{\alpha,m}^{T}$ from the left collects the $1$-st, the $(\alpha+1)$-th, \ldots, the $((m-1)\alpha+1)$-th row together, further the $2$-nd, the $(\alpha+2)$-th, \ldots, the $((m-1)\alpha+2)$-th row together, and soforth.

Suppose $X=[X_{rs}]_{r,s=1}^{N}$ is as in Theorem \ref{posls} (\ref{posls2}). Next, fix $r,s$ and let $b=\min\{\alpha_r,\alpha_s\}$. 
Denote the block of $X_{rs}$ in the $j$-th row and the $k$-th column by 
\[
(X_{rs})_{jk}=
\left\{
\begin{array}{ll}
[0\quad T_{jk}], & \alpha_{r}<\alpha_{s}\\
\begin{bmatrix}
T_{jk}\\
0
\end{bmatrix}, & \alpha_{r}>\alpha_{s}\\
T_{jk},& \alpha_{r}=\alpha_{s}
\end{array}\right., \qquad
\begin{array}{l}
j\in \{1,\ldots m_r\},\quad k\in \{1,\ldots m_s\},\\
T_{jk}:=T(a_0^{jk},a_1^{jk},\ldots,a_{b-1}^{jk}).
\end{array}
\]
By setting $A_n:=[a_n^{jk}]_{j,k=1}^{m_r,m_s}\in \mathbb{C}^{m_r\times m_r}$ for $n\in \{0,\ldots,b-1\}$ and $\mathcal{T}=T(A_0,\ldots,A_{b-1})$, we obtain 
a rectangular block upper-triangular Toeplitz matrix of size $\alpha_r\times \alpha_s$:
%
\begin{equation*}
\Omega_{\alpha_r,m_r}^{T}X_{rs}\Omega_{\alpha_s,m_s}=
\left\{
\begin{array}{ll}
[0\quad \mathcal{T}], & \alpha_r<\alpha_s\\
\begin{bmatrix}
\mathcal{T}\\
0
\end{bmatrix}, & \alpha_r>\alpha_s\\
\mathcal{T},& \alpha_r=\alpha_s
\end{array}\right..
\end{equation*}
%
%
Thus we get a matrix of the form (\ref{0T0}):
\begin{align}\label{otxo}
\Omega^{T}X\Omega=  &  
[\Omega_{\alpha_r,m_r}^{T}X_{rs}\Omega_{\alpha_s,m_s}]_{r,s=1}^{N}, \qquad (\Omega:=\oplus_{r=1}^{N}\Omega_{\alpha_r,m_r}).
\end{align}

\begin{example}
$N=2$, $\alpha_1=3$, $m_1=2$, $\alpha_2=2$, $m_2=3$:
\[
\Omega_{3,2}^{T}\begin{bmatrix}[cc|cc|cc]
a_1 & b_1 & a_2 & b_2 & a_3 & b_3 \\
0   & a_1 & 0   & a_2 & 0   & a_3\\
0   & 0   & 0   & 0   & 0   &  0 \\
\hline
a_4 & b_4 & a_5 & b_5 & a_6 & b_6 \\
0   & a_4 & 0   & a_5 & 0   & a_6\\
0   & 0   & 0   & 0   & 0   &  0 
\end{bmatrix}\Omega_{2,3}
=
\begin{bmatrix}[ccc|ccc]
a_1 & a_2 & a_3 & b_1 & b_2 & b_3 \\
a_4 & a_5 & a_6   & b_4 & b_5   & b_6\\
\hline
0   & 0   & 0   & a_1   & a_2   &  a_3 \\
0 &  0 &   0 & a_4 & a_5 & a_6 \\
\hline
0   & 0 & 0   & 0 & 0   & 0\\
0   & 0   & 0   & 0   & 0   &  0 
\end{bmatrix}.
\]
\end{example}

Next, we observe that the set of nonsingular matrices of the form (\ref{0T0}) has a special group structure, similar to the group of all nonsingular upper-triangular matrices. 
We use ideas from the proof of a similar (maybe somewhat stronger) result for upper-unitriangular matrices \cite[Proposition 3.31]{CK}, \cite[Example 6.49]{Milne}. 

\begin{lemma}\label{lemanilpo}
Let $\mathbb{T}$ be the set of all nonsingular matrices of the form (\ref{0T0}). Then $\mathbb{T}$ is a subgroup of the group of all nonsingular matrices. Furthermore, $\mathbb{T}=\mathbb{D}\ltimes \mathbb{U}$ is a semidirect product of subgroups, where $\mathbb{D}$ contains all nonsingular block-diagonal matrices and $\mathbb{U}$ is a normal subgroup that consists of matrices whoose diagonal blocks are block upper-triangular Toeplitz matrices with identity as the diagonal block. Further, $\mathbb{U}$ is unipotent of order at most $\alpha_1-1$ and it has has nilpotency class at most $\alpha_1$.
\end{lemma}

\begin{proof}
First, we examine the set $\mathbb{U}$ of all nonsingular matrices of the form (\ref{0T0}) such that their diagonal blocks are block upper-triangular Toeplitz matrices with identities as the diagonal blocks.

%

For 
$k\in \{1,\ldots,\alpha_1-1\}$ let $\mathfrak{N}_k$ be the set of nonsingular matrices 
of the form (\ref{0T0}) 
with $\mathcal{T}_{rs}=T(0,\ldots,0,A_{k}^{rs},A_{k+1}^{rs},\ldots,A_{b_{rs}-1}^{rs})$ (i.e. $A_0^{rs}=\ldots=A_{k-1}^{rs}=0$) for $b_{rs}>k$ and $\mathcal{T}_{rs}=0$ for $k\geq b_{rs}$, and such that all $A_{k}^{rr}=0$.
We have 
\[
\mathbb{U}-\mathcal{I}=:\mathfrak{N}_0\supset\mathfrak{N}_1\supset \cdots \supset \mathfrak{N}_{\alpha_{1}-1}=\{0\}.
\]
Sums and products of rectangular upper-triangular Toeplitz matrices of the appropriate size are again rectangular upper-triangular Toeplitz matrices. Moreover,
\[
\mathfrak{N}_k+\mathfrak{N}_k\subset \mathfrak{N}_{k}, \qquad
\mathfrak{N}_0\mathfrak{N}_k\subset \mathfrak{N}_{k+1}, \qquad 
\mathfrak{N}_k\mathfrak{N}_0\subset \mathfrak{N}_{k+1}.
\]
In particular $\mathfrak{N}_k^{\alpha_1-k-1}=\{0\}$, thus matrices in $\mathfrak{N}_k$ are nilpotent. For $\mathcal{N}\in \mathfrak{N}_k$ we have 
\[
(\mathcal{I}+\mathcal{N})^{-1}=\mathcal{I}-\mathcal{N}+\mathcal{N}^{2}-\ldots+(-1)^{\alpha_1-k-1}\mathcal{N}^{\alpha_1-k-1}.  
\]
Hence $\mathbb{U}_k:=\mathcal{I}+\mathfrak{N}_k$ is a unipotent group.
Taking $\mathcal{I}+\mathcal{N}\in \mathbb{U}_k$ and $\mathcal{I}+\mathcal{N}'\in \mathbb{U}$ with $(\mathcal{I}+\mathcal{N}')^{-1}=\mathcal{I}-\mathcal{N}'+(\mathcal{N}')^{2}-\ldots$ we get
\[
(\mathcal{I}+\mathcal{N}')^{-1}(\mathcal{I}+\mathcal{N})(\mathcal{I}+\mathcal{N}')=\mathcal{I}+\big((\mathcal{I}-\mathcal{N}'+(\mathcal{N}')^{2}-\ldots)\mathcal{N}(\mathcal{I}+\mathcal{N}')\big)\in \mathbb{U}_k,
\]
and the commutator is of the form
\begin{align*}
[\mathcal{I}+\mathcal{N},\mathcal{I}+\mathcal{N}']
&=(\mathcal{I}+\mathcal{N})^{-1}(\mathcal{I}+\mathcal{N}')^{-1}(\mathcal{I}+\mathcal{N})(\mathcal{I}+\mathcal{N}')\\
&=\big((\mathcal{I}-\mathcal{N}+\mathcal{N}^{2}-\ldots)(\mathcal{I}-\mathcal{N}'+(\mathcal{N}')^{2}-\ldots)\big)\big((\mathcal{I}+\mathcal{N})(\mathcal{I}+\mathcal{N}')\big)\\
&=(\mathcal{I}-\mathcal{N}-\mathcal{N}'+\mathcal{M}_1)(\mathcal{I}+\mathcal{N}+\mathcal{N}'+\mathcal{M}_2)\\
&=\mathcal{I}+\mathcal{M}_3\in \mathbb{U}_k,
\end{align*}
where $\mathcal{M}_1,\mathcal{M}_2,\mathcal{M}_3\in\mathfrak{N}_{k+1}$.
Hence 
%
\begin{equation}\label{cseq}
\mathbb{U}=\mathbb{U}_0\supset \mathbb{U}_1\supset \cdots \supset \mathbb{U}_{\alpha_{1}-1}
=\{\mathcal{I}\}
\end{equation}
is a central series of normal subgroups, i.e. $[\mathbb{U},\mathbb{U}_{j}]$ is a commutator group of $ \mathbb{U}_{j+1}$.
%

Any $\mathcal{X}\in \mathbb{T}$ (nonsingular and of the form (\ref{0T0})) can be written as $\mathcal{X}=\mathcal{D}\mathcal{U}$, where $\mathcal{U}\in \mathbb{U}$ and $\mathcal{D}\in \mathbb{D}$ is a nonsingular block-diagonal matrix of the form (\ref{0T0}). For $\mathcal{D}_1,\mathcal{D}_2\in \mathbb{D}$ and $\mathcal{U}_1,\mathcal{U}_2\in \mathbb{U}$ we get that $(\mathcal{D}_1\mathcal{U}_1)(\mathcal{D}_2\mathcal{U}_2)^{-1}=\mathcal{D}_1(\mathcal{U}_1\mathcal{U}_2^{-1})\mathcal{D}_2^{-1}$ is of the form (\ref{0T0}), thus $\mathbb{T}$ is a group.
Next, conjugating $\mathcal{I}+\mathcal{N}\in\mathcal{I}+\mathfrak{N}_0=\mathbb{U}$ by $\mathcal{D}\mathcal{U}$ gives
\[
\mathcal{U}^{-1}\mathcal{D}^{-1}(\mathcal{I}+\mathcal{N})\mathcal{D}\mathcal{U}=\mathcal{I}+\mathcal{U}^{-1}\mathcal{D}^{-1}\mathcal{N}\mathcal{D}\mathcal{U} \in \mathbb{U},
\]
This proves normality of $\mathbb{U}$ and concludes the proof.
\end{proof}

\begin{remark}
It would be interesting to know whether (\ref{cseq}) is a lower central sequence or not. Note that the situation seems more involved than in the case of upper-unitriangular matrices, in which the commutators of suitably chosen generators are again generators (see \cite{CK}). 
\end{remark}

\section{Certain block matrix equation}\label{cereq}

%
In this section we consider certain block upper-triangular Toeplitz matrix equation. 
Its solution (Lemma \ref{EqT}) is the key ingredient in the proof of Theorem \ref{stabs}.

Let $\alpha_{1}>\alpha_{2}>\ldots >\alpha_{N}$ and $m_1,\ldots,m_N\in \mathbb{N}$. Suppose 
\begin{align}\label{BBF}
&\mathcal{B}=\bigoplus_{r=1}^{N}T\big(B_0^{r},B_1^{r},\ldots,B_{\alpha_r-1}^{r}\big),\,\, \mathcal{C}=\bigoplus_{r=1}^{N}T\big(C_0^{r},C_1^{r},\ldots,C_{\alpha_r-1}^{r}\big),\,\, \mathcal{F}=\bigoplus_{r=1}^{N}E_{\alpha_r}(I_{m_r}),\\
& B_0^{r},C_0^{r}\in GL_{m_r}(\mathbb{C})\cap S_{m_r}(\mathbb{C}), \qquad B_1^{r},C_1^{r},\ldots, B_{\alpha_r-1}^{r}, C_{\alpha_r-1}^{r}\in S_{m_r}(\mathbb{C}),\nonumber
\end{align}
%
where 
$E_{\alpha}(I_{m})=
\begin{bsmallmatrix}
 0                 &      & I_{m}\\
            &   \iddots     &  \\
I_{m}            &           &  0\\
\end{bsmallmatrix}
$
is an $\alpha\times \alpha$ block matrix with $I_m$ on the anti-diagonal and zero-matrices otherwise.
%
We shall solve a matrix equation
\begin{equation}\label{eqFYFIY}
\mathcal{C}=\mathcal{F}\mathcal{X}^{T}\mathcal{F}\mathcal{B} \mathcal{X},
\end{equation}
where $\mathcal{X}=[\mathcal{X}_{rs}]_{r,s=1}^{N}$ is of the form as in (\ref{0T0}).
%
%
%

We first observe a few simple facts. 
%
%
%
%
%
The calculation
\[
(\mathcal{F}\mathcal{X}^{T}\mathcal{F}\mathcal{B} \mathcal{X})^{T}=\mathcal{X}^{T}\mathcal{B}^{T}\mathcal{F}\mathcal{X}\mathcal{F} =\mathcal{F}\mathcal{F}\mathcal{X}^{T}\mathcal{F}(\mathcal{F}\mathcal{B}^{T}\mathcal{F})\mathcal{X}\mathcal{F}=\mathcal{F}(\mathcal{F}\mathcal{X}^{T}\mathcal{F}\mathcal{B}\mathcal{X})\mathcal{F}
\]
shows that for $r\neq s$ we have $(\mathcal{F}\mathcal{X}^{T}\mathcal{F}\mathcal{B}\mathcal{X})_{rs}=0$ if and only if $(\mathcal{F}\mathcal{X}^{T}\mathcal{F}\mathcal{B}\mathcal{X})_{sr}=0$.
When comparing the left-hand side with the right-hand side of (\ref{eqFYFIY}) blockwise, it thus suffices to observe only blocks in the upper-triangular parts of $\mathcal{F}X^{T}\mathcal{F}\mathcal{B} X$ and $\mathcal{C}$.  
Since $(\mathcal{F}\mathcal{X}^{T}\mathcal{F}\mathcal{B}\mathcal{X})_{rs}$ and $\mathcal{C}_{rs}$ are rectangle block Toeplitz and of the same form for each $r, s$,
it is enough to compare the first rows of these blocks.

The following lemma explains the process of computing solutions of (\ref{eqFYFIY}). 
In the proof of Theorem \ref{stabs} we shall obtain (\ref{eqFYFIY}) for $\mathcal{B}$ and $\mathcal{C}$ equal to the identity-matrix. However, due to a possible application when computing isotropy groups of actions similar to (\ref{aos}) and since it makes no serious difference to the proof, we prove a little more general result. 

\vspace{-1mm}

\begin{lemma}\label{EqT}
Let $\mathcal{B}$, $\mathcal{C}$ as in (\ref{BBF}) be given.
Then the dimension of the space of solutions
of (\ref{eqFYFIY}) that are of the form
$\mathcal{X}=[\mathcal{X}_{rs}]_{r,s=1}^{N}$ (partitioned conformally to $\mathcal{B}$, $\mathcal{C}$) with
\begin{equation}\label{EqTX}
\mathcal{X}_{rs}=
\left\{
\begin{array}{ll}
[0\quad \mathcal{T}_{rs}], & \alpha_r<\alpha_s\\
\begin{bmatrix}
\mathcal{T}_{rs}\\
0
\end{bmatrix}, & \alpha_r>\alpha_s\\
\mathcal{T}_{rs},& \alpha_r=\alpha_s
\end{array}\right., \quad 
\begin{array}{l}
\alpha_{1}>\alpha_{2}>\ldots >\alpha_{N}, \quad b_{rs}:=\min\{\alpha_s,\alpha_r\}\\
\mathcal{T}_{rs}=T\big(A_0^{rs},A_1^{rs},\ldots,A_{b_{rs}-1}^{rs}\big),\quad A_j^{rs}\in \mathbb{C}^{m_r\times m_s}
\end{array}
\end{equation}
%
%
%
is $\sum_{r=1}^{N}\alpha_r m_r\big(\tfrac{1}{2}(m_r-1)+\sum_{s=1}^{r-1}m_s\big)$. In particular, the general solution
satisfies the following properties:
\begin{enumerate}[label={(\alph*)},ref={\alph*},itemsep=5pt,leftmargin=25pt,itemindent=0pt]
\item \label{EqT2} The entries of $A_0^{rr}$ for $r\in \{1,\ldots,N\}$ can be taken so that $A_0^{rr}$ is any solution of the equation $C_0^{r}=(A_0^{rr})^{T}B_0^{r}A_0^{rr}$.
If $N\geq 2$ the entries of $A_j^{rs}$ for $r,s\in \{1,\ldots,N\}$ with $r>s$ and $j\in \{0,\ldots,\alpha_{r}-1\}$ can be taken as free variables. 
\item \label{EqT3} 
Assuming (\ref{EqT2}) and choosing the entries of matrices $Z_j^{r}=-Z_j^{r}\in \mathbb{C}^{m_r\times m_r}$ for $r\in \{1,\ldots,N\}$, $\alpha_r-1\geq j\geq 1$ as free variables, the remaining entries of $\mathcal{X}$ are computed by the following algorithm:

\vspace{-2mm}
\hspace{-10mm}

\begin{algorithmic}
\State $\Psi_n^{krs}:=\sum_{j=0}^{n} \sum_{l=0}^{n-j}(A_j^{kr})^{T} B_{n-j-l}^{k}A_l^{ks}$
\For {$j=0:\alpha_1-1$}
    \If {$r\in \{1,\ldots,N\}$, $1\leq j\leq \alpha_r-1$}
    \State $A_j^{rr}=A_0^{rr}-A_0^{rr}(C_0^{r})^{-1}\big(Z_j^{r}+\sum_{l=1}^{j-1}\sum_{m=0}^{n-j}( A_{l}^{rr})^{T} B_{n-j-m}^{k}A_m^{kr}$\\
    \qquad \qquad \qquad \qquad \qquad \qquad \quad $+\sum_{k=1}^{r-1}\Psi_{j-\alpha_k+\alpha_r}^{rr}+\sum_{k=r+1}^{N}\Psi_{j-\alpha_{r}+\alpha_k}^{rr}\big)$
    \EndIf
    \For {$p=1:N-1$}
        \If {$r\in \{1,\ldots,N\}$, $j\leq \alpha_{r+p}-1$, $r+p\leq N$}
        \small
        \State \hspace{-3mm} $A_j^{r(r+p)}=-A_0^{r(r+p)}(C_0^{r})^{-1}\big(
        \sum_{l=1}^{j}\sum_{m=0}^{n-j}( A_{l}^{rr})^{T} B_{n-j-m}^{k}A_m^{k(r+p)}+\sum_{k=1}^{r-1}\Psi_{j-\alpha_k+\alpha_r}^{r(r+p)}$\\
        \qquad \qquad \qquad \qquad \qquad \qquad \qquad \qquad  $+ \sum_{k=r+1}^{r+p}\Psi_{j}^{r(r+p)}+\sum_{k=r+p+1}^{N}\Psi_{j-\alpha_{r+p}+\alpha_k}^{r(r+p)}\big),$   
\normalsize
        \EndIf
    \EndFor
\EndFor
\end{algorithmic}
For simplicity, in this algorithm we define $\sum_{j=l}^{n}a_j=0$ if $l>n$, and it is understood that the inner loop (i.e. for p =1 : N-1) is not performed for $N=1$. 
\end{enumerate}

\vspace{-1mm}

Furthermore, assume that $\mathcal{B}$ and $\mathcal{C}$ are real. Then the solution $\mathcal{X}$ is real 
if and only if 
\begin{enumerate}[label={(\roman*)},ref={\roman*},itemsep=2pt,leftmargin=25pt,itemindent=0pt]
\item Matrices $B_0^{r}$ and $C_0^{r}$ in (\ref{BBF}) have the same inertia
for all $r\in \{1,\ldots,N\}$.
\item Matrices $A_j^{rs}$ with $r>s$, $j\in \{0,\ldots,\alpha_{r}-1\}$, $N\geq 2$, matrices $A_0^{rr}$, and matrices $Z_j^{r}$ for $1 \leq j\leq \alpha_r-1$ ($r,s\in \{1,\ldots,N\}$) in (\ref{EqT2}) and (\ref{EqT3}) are chosen real.
\end{enumerate}
\end{lemma}

\vspace{-1mm}

For the sake of clarity we point out 
the importance of
%
the correct order of calculating the entries of $\mathcal{X}$ in Lemma \ref{EqT}. It is essential for the proof of the lemma.
%
%
%
%
Recall first that by (\ref{EqT2}) (when $N\geq 2$) all entries of the blocks below the main diagonal of $\mathcal{X}=[\mathcal{X}_{rs}]_{r,s=1}^{N}$ can be chosen freely ($\sum_{r=1}^{N}\sum_{s=1}^{r-1}\alpha_r m_rm_s$ free variables).
Next, we compute the diagonal entries of the blocks in the upper triangular part of $\mathcal{X}$. We first obtain the diagonal entries of the main diagonal blocks $\mathcal{X}_{rr}$ for $r\in \{1,\ldots,N\}$; they add $\sum_{r=1}^{N}\frac{1}{2}m_r(m_r-1)$ to the dimension of the solution space (see (\ref{EqT2}) again). Secondly, step $j=0$, $p=1$ (if $N\geq 2$) of the algorithm in (\ref{EqT3}) yields the diagonal entries of the first upper off-diagonal blocks of $\mathcal{X}$ (i.e. $(\mathcal{X}_{r(r+1)})_{11}=A_{0}^{r(r+1)}$). Further, step $j=0$, $p=2$ gives the diagonal entries of the second upper off-diagonal blocks of $\mathcal{X}$ (i.e. $(\mathcal{X}_{r(r+2)})_{11}=A_{0}^{r(r+2)}$), step $j=0$, $p=3$ gives the diagonal entries of the third upper off-diagonal blocks of $\mathcal{X}$ (i.e. $(\mathcal{X}_{r(r+3)})_{11}=A_{0}^{r(r+3)}$), and soforth. In the same fashion 
the step for fixed $j\in \{1,\ldots,\alpha_1-1\}$ and $p\in \{0,\ldots,N\}$ yields the entries on the $j$-th upper off-diagonals of the $p$-th upper off-diagonal blocks of $\mathcal{X}$ (i.e. $(\mathcal{X}_{r(r+p)})_{1(j+1)}=A_{j+1}^{r(r+p)}$ with $r+p\leq N$, provided that $j\leq \alpha_{r+p}-1$). Finally, at step $j=\alpha_1-1$, $p=0$ we compute $(\mathcal{X}_{11})_{1\alpha_1}=A_{\alpha_1-1}^{11}$. Note that when calculating each entry $A_j^{rr}\in \mathbb{C}^{m_r\times m_r}$,
we add $\frac{1}{2}m_r(m_r-1)$ free variables. 
Furthermore, this algorithmic procedure allows to compute each entry from the entries that are already known.


\begin{proof}[Proof of Lemma \ref{EqT}]
%
%
%
The idea is to write the equation (\ref{eqFYFIY}) entrywise as a system of several simpler matrix equations and then consider them in an appropriate order.

First, we analyze the right-hand side of the equation (\ref{eqFYFIY}) for $\mathcal{B}$, $\mathcal{F}$ of the form (\ref{BBF}) and $\mathcal{X}=[\mathcal{X}_{rs}]_{r,s=1}^{N}$ with blocks as in (\ref{EqTX}). To simplify the notation we set $\mathcal{Y}:=\mathcal{B}\mathcal{X}$ and $\widetilde{\mathcal{X}}:=\mathcal{F}X^{T}\mathcal{F}$. 
The entries in the $j$-th column and in the first row of $(\widetilde{\mathcal{X}}\mathcal{Y})_{rs}$ 
are obtained by multiplying the first rows of the blocks $\widetilde{\mathcal{X}}_{r1},\ldots, \widetilde{\mathcal{X}}_{rN}$ with the $j$-th columns of the blocks $(\mathcal{Y})_{1s},\ldots, (\mathcal{Y})_{Ns}$, respectively, and then adding them:
\begin{equation}\label{YSrs1j}
((\widetilde{\mathcal{X}}\mathcal{Y})_{rs})_{1j}=\sum_{k=1}^{N}(\widetilde{\mathcal{X}}_{rk})_{(1)}(\mathcal{Y}_{ks})^{(j)}, \qquad r,s\in \{1,\ldots,N\}, \quad j\in \{1,\ldots,\alpha_s\}.
\end{equation}
As mentioned in the discussion in the beginning of this section it suffices analyse the upper-triangular blocks of $\widetilde{\mathcal{X}}\mathcal{Y}$: 
%
\vspace{-1mm}
\begin{align*}
((\widetilde{\mathcal{X}}\mathcal{Y})_{r(r+p)})_{1j}
=  \sum_{k=1}^{N}(\widetilde{\mathcal{X}}_{rk})_{(1)}(\mathcal{Y}_{k(r+p)})^{(j)}, \qquad 1\leq j\leq \alpha_{r+p},\quad 0\leq p \leq N-r.
\end{align*}
\vspace{-1mm}
When $N=1$ (hence $r=1$, $p=0$) we have
%
\begin{align}\label{f00}
((\widetilde{\mathcal{X}}\mathcal{Y})_{11})_{1j}
 =
 (\widetilde{\mathcal{X}}_{11})_{(1)}((\mathcal{Y})_{11})^{(j)},
\end{align}
\vspace{-1mm}
while for $N\geq 2$ we obtain 
\vspace{-1mm}
%
%
\begin{align}\label{YSp1}
((\widetilde{\mathcal{X}}\mathcal{Y})_{1(1+p)})_{1j}= & (\widetilde{\mathcal{X}}_{11})_{(1)}(\mathcal{Y}_{1(1+p)})^{(j)}+\sum_{k=2}^{N}(\widetilde{\mathcal{X}}_{1k})_{(1)}(\mathcal{Y}_{k(1+p)})^{(j)}, \qquad (r=1)\\
%
\vspace{-1mm}
\label{f1}
((\widetilde{\mathcal{X}}\mathcal{Y})_{r(r+p)})_{1j}= & (\widetilde{\mathcal{X}}_{rr})_{(1)}(\mathcal{Y}_{r(r+p)})^{(j)}+\sum_{k=r+1}^{N}(\widetilde{\mathcal{X}}_{rk})_{(1)}(\mathcal{Y}_{k(r+p)})^{(j)}\\
\vspace{-1mm}
  &+\sum_{k=1}^{r-1}(\widetilde{\mathcal{X}}_{rk})_{(1)}(\mathcal{Y}_{k(r+p)})^{(j)}, \qquad 1<r<N,  \nonumber\\
\vspace{-1mm}
\label{fN}
((\widetilde{\mathcal{X}}\mathcal{Y})_{NN})_{1j}
= & (\widetilde{\mathcal{X}}_{NN})_{(1)}(\mathcal{Y}_{NN})^{(j)}+ \sum_{k=1}^{N-1}(\widetilde{\mathcal{X}}_{Nk})_{(1)}(\mathcal{Y}_{kN})^{(j)}, \qquad (r=N).
\end{align}
\vspace{-1mm}
For any $r,k\in \{1,\ldots,N\}$ we have
%
%
\begin{align}\label{ETE}
&E_{a_{rk}}(I_{m_k})\big(T(A_0,A_1,\ldots,A_{b_{rk}-1})\big)^{T}E_{a_{rk}}(I_{m_r})=T(A_0^{T},A_1^{T},\ldots,A^{T}_{b_{rk}-1}),
\end{align}
and it further implies
%
\begin{align}\label{AkrT}
\widetilde{\mathcal{X}}_{rk}
=
E_{\alpha_r}(I_{m_r})\mathcal{X}_{kr}^{T} E_{\alpha_k}(I_{m_k})
&=
\left\{
\begin{array}{cc}
\begin{bsmallmatrix}
\widetilde{\mathcal{T}}_{rk}\\
0
\end{bsmallmatrix}, 
& \alpha_r >\alpha_k \\
\begin{bsmallmatrix}
0 & \widetilde{\mathcal{T}}_{rk}
\end{bsmallmatrix}, & \alpha_r<\alpha_k\\
\widetilde{\mathcal{T}}_{rk}, & \alpha_r=\alpha_k
\end{array}
\right., \quad
\widetilde{T}_{rk}=
T\big((A_0^{kr})^{T},\ldots,(A_{b_{kr}}^{kr})^{T}\big),\nonumber\\
%
%
(\widetilde{\mathcal{X}}_{rk})_{(1)}
&=\left\{
\begin{array}{ll}
\begin{bsmallmatrix}
(A_0^{kr})^{T} & (A_1^{kr})^{T} & \ldots & (A_{a_{k}-1}^{kr})^{T} 
\end{bsmallmatrix}, & \alpha_k\leq \alpha_r\\
\begin{bsmallmatrix}
0 & \ldots & 0 & (A_0^{kr})^{T} & \ldots & (A_{\alpha_r-1}^{kr})^{T}
\end{bsmallmatrix}, & \alpha_k > \alpha_r
\end{array}
\right..
\end{align}
We define $\Phi_{n}^{ks}:=\sum_{j=0}^{n} B_{n-j}^{k}A_j^{ks}$ and observe that
\begin{align}\label{YSP}
&\mathcal{Y}_{ks}
=T\big(B_0^{k},B_1^{k},\ldots,B_{\alpha_k-1}^{k}\big)
\left\{
\begin{array}{ll}
\begin{bsmallmatrix}
\mathcal{T}_{ks}\\
0
\end{bsmallmatrix}, 
& \alpha_k >\alpha_s \\
\begin{bsmallmatrix}
0 & \mathcal{T}_{ks}
\end{bsmallmatrix}, & \alpha_k<\alpha_s\\
\mathcal{T}_{ks}, & \alpha_k=\alpha_s
\end{array}
\right.
=
\left\{
\begin{array}{ll}
\begin{bsmallmatrix}
\mathcal{S}_{ks}\\
0
\end{bsmallmatrix}, 
& \alpha_k >\alpha_s \\
\begin{bsmallmatrix}
0 & \mathcal{S}_{ks}
\end{bsmallmatrix}, & \alpha_k<\alpha_s\\
\mathcal{S}_{ks}, & \alpha_k=\alpha_s
\end{array}
\right., \\
&\mathcal{S}_{ks}=
T\big(B_0^{k},B_1^{k},\ldots, B_{b_{ks}-1}^{k}\big)T\big(A_0^{ks},A_1^{ks},\ldots,A_{b_{ks}-1}^{ks}\big)
=T\big(\Phi_0^{ks},\Phi_1^{ks},\ldots, \Phi_{b_{ks}-1}^{ks}\big).\nonumber
\end{align}
%
%
%
%
%

We begin with the calculation of matrices $A_0^{rr}$ for $r\in \{1,\ldots,N\}$. Since
\[
(\widetilde{\mathcal{X}}_{rk})_{(1)}=\left\{
\begin{array}{ll}
\begin{bsmallmatrix}
(A_0^{kr})^{T} & * & \ldots & *
\end{bsmallmatrix}, & k\geq r\\
\begin{bsmallmatrix}
0 & * & \ldots & *
\end{bsmallmatrix}, & k<r
\end{array}
\right.,\qquad
(\mathcal{Y}_{kr})^{(1)}=
\left\{
\begin{array}{ll}
\begin{bsmallmatrix}
B_0^{k}A_0^{kr} \\
0\\
\vdots\\
0
\end{bsmallmatrix}, & k\leq r\\
0, & k>r
\end{array}
\right.,
\]
we get from (\ref{YSrs1j}) for $r=s$, $j=1$ that
\begin{align*}
((\widetilde{\mathcal{X}}\mathcal{Y})_{rr})_{11}=\sum_{k=1}^{N}(\widetilde{\mathcal{X}}_{rk})_{(1)}((\mathcal{Y})_{kr})^{(1)}=(\widetilde{\mathcal{X}}_{rr})_{(1)}(\mathcal{Y}_{rr})^{(1)}=(A_0^{rr})^{T}B_0^{r}A_0^{rr},\quad r\in \{1,\ldots,N\}.
\end{align*}
Together with $(\mathcal{C}_{rr})_{11}=C_0^{r}$ the equation (\ref{eqFYFIY}) yields an equation that gives $A_0^{rr}$:
\begin{equation}\label{GABA}
C_0^{r}=(A_0^{rr})^{T}B_0^{r}A_0^{rr}, \qquad r\in \{1,\ldots,N\}.
\end{equation}
%
%
Next, if $N\geq 2$, we fix arbitrarily the blocks below the main diagonal of $[\mathcal{X}_{rs}]_{r,s=1}^{N}$ (hence the blocks above the main 
diagonal of  $[\widetilde{\mathcal{X}}_{rs}]_{r,s=1}^{N}$). This corresponds to (\ref{EqT2}).
%

Proceed with the key step in the proof: an inductive procedure that enables to compute the remaining entries (i.e. the algorithm in (\ref{EqT3})).
We fix $r\in \{1,\ldots,N\}$, $p\in \{0,\ldots, N-r \}$ and $j\leq \alpha_r-1$, but not $p=j=0$. Assuming that we have already determined the matrices 
$A_{\widetilde{j}}^{r's'}$ (with $1 \leq r', s'\leq N$) for 
%
\begin{align}\label{induA}
&j\geq 1, \widetilde{j}\in \{0,\ldots,j-1\}, s'\geq r' 
\quad\textrm{or}\quad p\geq 1, \widetilde{j}=j, r'\leq s'\leq r'+p-1 \\
&\textrm{ or } s'\leq r', \widetilde{j}\in \{0,\ldots,b_{r's'}-1\}, N\geq 2 \nonumber
\end{align}
%
we shall compute $A_j^{r(r+p)}$.
Essentially, we shall solve the equation $(\mathcal{C}_{r(r+p)})_{1j}=((\widetilde{\mathcal{X}}\mathcal{Y})_{r(r+p)})_{1j}$ (see (\ref{eqFYFIY})). By a careful analysis of the structures of  
$(\widetilde{\mathcal{X}}_{rk})_{(1)}$ and $(\mathcal{Y}_{k(r+p)})^{(j)}$ in formulas (\ref{YSp1}), (\ref{f1}), (\ref{fN}), we shall reduce this equation to a simple linear matrix equation in $A_j^{r(r+p)}$ (and possibly $(A_j^{r(r+p)})^{T}$) with coefficients depending only on $A_{\widetilde{j}}^{r's'}$ for (\ref{induA}).

For the sake of clarity we set the notation ($n\in \mathbb{Z}$, $k,r,s\in \{1,\ldots,N\}$):
\begin{align}\label{Prsk}
\small
\Psi^{krs}_{n}:=
\left\{
\begin{array}{ll}
\begin{bsmallmatrix}
(A_0^{kr})^{T} & (A_1^{kr})^{T} & \ldots & (A_{n}^{rr})^{T} 
\end{bsmallmatrix}
\begin{bsmallmatrix}
\Phi_{n}^{ks} \\
\vdots \\
\Phi_{0}^{ks} \\
\end{bsmallmatrix}, & n\geq 0
\\
0, & n<0
\end{array}
\right.
=
\left\{
\begin{array}{ll}
\sum_{j=0}^{n}(A_j^{kr})^{T}\Phi_{n-j}^{ks}, & n\geq 0\\
0, & n<0
\end{array}
\right.. 
\normalsize
\end{align}
%
Note that:
\begin{align}\label{simetrija}
\Psi^{krs}_{n}
&=\sum_{j=0}^{n}(\Phi_j^{kr})^{T}A_{n-j}^{ks} 
=\sum_{j=0}^{n}\sum_{l=0}^{j}(A_l^{kr})^{T}( B_{j-l}^{r})^{T}A_{n-j}^{ks}
=\sum_{l=0}^{n}\sum_{j=l}^{n}(A_l^{kr})^{T}( B_{j-l}^{r})^{T}A_{n-j}^{ks}\nonumber\\
&=\sum_{l=0}^{n}\sum_{j'=0}^{n-l}(A_{l}^{kr})^{T}( B_{j'}^{r})^{T}A_{n-l-j'}^{ks}
=\sum_{l=0}^{n}(A_l^{kr})^{T}\Phi_{n-l}^{ks}=(\Psi^{ksr}_{n})^{T} .
\end{align}
%
%
%

Since
$(\widetilde{\mathcal{X}}_{rr})_{(1)}=
\begin{bsmallmatrix}
(A_0^{rr})^{T} &  \ldots & (A_{\alpha_{r}-1}^{rr})^{T}
\end{bsmallmatrix}
$
and
\begin{align*}
&(\mathcal{Y}_{rr})^{(\alpha_r-1)}=
\begin{bsmallmatrix}
\Phi_{\alpha_r-1}^{rr} \\
\vdots\\
\Phi_0^{rr}
\end{bsmallmatrix},\qquad
(\mathcal{Y}_{r(r+p)})^{(j+1)}=
\begin{bsmallmatrix}
\Phi_j^{r(r+p)} \\
\vdots\\
\Phi_0^{r(r+p)}\\
0\\
\vdots\\
0
\end{bsmallmatrix},\quad  j< \alpha_s-1 \textrm{ or } p\geq 1,
\end{align*}
the first term of (\ref{f00}), (\ref{YSp1}), (\ref{f1}), (\ref{fN}) is:
%
\begin{align}\label{xijrp}
&(\widetilde{\mathcal{X}}_{rr})_{(1)}(\mathcal{Y})_{rr}^{(j+1)}=\Psi_j^{rrr}=(A_0^{rr})^{T}B_0^{r}A_j^{rr}+(A_j^{rr})^{T}B_0^{r}A_0^{rr}+\Xi(j,r,0),\quad (p=0)\nonumber\\
&(\widetilde{\mathcal{X}}_{rr})_{(1)}(\mathcal{Y})_{r(r+p)}^{(j+1)}=\Psi_j^{rr(r+p)}=
(A_0^{rr})^{T}B_0^{r}A_j^{r(r+p)}+\Xi(j,r,p),\qquad p\geq 1,\nonumber\\
%
&\qquad \qquad
\Xi(j,r,p):=
\left\{
\begin{array}{ll}
\sum_{l=1}^{j-1}A_{l}^{rr}\Phi_{j-l}^{rr}, & j \geq 1, p=0 \\
\sum_{l=1}^{j} A_{l}^{r}\Phi_{j-l}^{r(r+p)}, & j \geq 0, p\geq 1 \\
\end{array}
\right..
\end{align}
%
%
(For simplicity we have defined $\sum_{l=1}^{n}a_l=0$ for $n<l$.)

When $N\geq 2$ the second term in (\ref{YSp1}) and (\ref{f1}) for $j+1$ instead of $j$ consists of summands $(\widetilde{\mathcal{X}}_{rk})_{(1)}(\mathcal{Y}_{k(r+p)})^{(j+1)}$ with $k\geq r+1$ and such that
%
\begin{align*}
&(\widetilde{\mathcal{X}}_{rk})_{(1)}=
\begin{bsmallmatrix}
(A_0^{kr})^{T} &  \ldots & (A_{\alpha_{r}-1}^{kr})^{T}
\end{bsmallmatrix},\\ 
%
\small
& (\mathcal{Y}_{(r+p)(r+p)})^{(j+1)}=
\begin{bsmallmatrix}
\Phi_j^{r(r+p)} \\
\vdots\\
\Phi_0^{r(r+p)}
\end{bsmallmatrix},  \qquad
(\mathcal{Y}_{k(r+p)})^{(j+1)}=
\left\{\begin{array}{ll}
\begin{bsmallmatrix}
\Phi_j^{r(r+p)} \\
\vdots\\
\Phi_0^{r(r+p)}\\
0\\
\vdots\\
0
\end{bsmallmatrix}, & r+p > k,\\
\begin{bsmallmatrix}
\Phi_{j-\alpha_{r+p}-\alpha_k}^{r(r+p)} \\
\vdots\\
\Phi_0^{r(r+p)}\\
0\\
\vdots\\
0
\end{bsmallmatrix}, & k> r+p
,
\end{array}
\right.
.
\end{align*}
\normalsize
Hence for $N\geq r+1\geq 2$:
\begin{align}\label{thjrp}
\Theta(j,r,p):=
& \sum_{k=r+1}^{N}(\widetilde{\mathcal{X}}_{rk})_{(1)}(\mathcal{Y}_{k(r+p)})^{(j+1)}\\
=
&\left\{\begin{array}{ll}
\sum_{k=r+1}^{N}\Psi_{j-\alpha_{r}+\alpha_k}^{krr}, & j\geq 1, p=0\\
\sum_{k=r+1}^{r+p}\Psi_{j}^{kr(r+p)}
+\sum_{k=r+p+1}^{N}\Psi_{j-\alpha_{r+p}+\alpha_k}^{kr(r+p)}, & j\geq 0, p\geq 1. 
\end{array}
\right.\nonumber
\end{align}
%
(For simplicity, we defined $\sum_{k=r+p+1}^{N}\Psi_{j-\alpha_{r+p}-\alpha_k}^{r(r+p)}=0$ for $r+p+1>N$.)

Finally, the third term in (\ref{f1}) and the second term in (\ref{fN}) (with $N\geq 2$) contain of summands which are products of matrices 
\[
(\widetilde{\mathcal{X}}_{rk})_{(1)}=
\begin{bsmallmatrix}
0& \ldots & 0 &(A_0^{kr})^{T}  & \ldots &(A_{b_{kr}}^{kr})^{T} 
\end{bsmallmatrix},\qquad
(\mathcal{Y}_{k(r+p)})^{(j+1)}=
\begin{bsmallmatrix}
\Phi_j^{k(r+p)} \\
\ldots\\
\Phi_0^{k(r+p)}\\
0\\
\ldots\\
0
\end{bsmallmatrix}, \quad 1\leq k\leq r-1,
\]
hence 
\begin{equation}\label{lajrp}
\Lambda(j,r,p):=\sum_{k=1}^{r-1}(\widetilde{\mathcal{X}}_{rk})_{(1)}(\mathcal{Y}_{k(r+p)})^{(j+1)}=\sum_{k=1}^{r-1}\Psi_{j-\alpha_k+\alpha_r}^{kr(r+p)}.
\end{equation}
%
%
%

We set the extensions by $0$:
\begin{align*}
&\widetilde{\Xi}(j,r,p)=\left\{\begin{array}{ll}
\Xi(j,r,p), & j\geq 2,p\geq 0\\
0,   &     \textrm{otherwise}
\end{array}
\right.,\qquad
\widetilde{\Theta}(j,r,p)=\left\{\begin{array}{ll}
\Theta(j,r,p), & N\geq r+1\geq 2\\
0,   &     \textrm{otherwise}
\end{array}
\right.\\
&\widetilde{\Lambda}(j,r,p)=\left\{\begin{array}{ll}
\Lambda(j,r,p), & N\geq r\geq 2\\
0,   &     \textrm{otherwise}
\end{array}
\right.
\end{align*}
and define
\begin{equation}\label{Djrp}
D_j^{r(r+p)}:=\widetilde{\Xi}(j,r,p)+\widetilde{\Theta}(j,r,p)+\widetilde{\Lambda}(j,r,p).
\end{equation}

The equation $(\mathcal{C}_{r(r+p)})_{1j}=((\widetilde{\mathcal{X}}\mathcal{Y})_{r(r+p)})_{1j}$ combined with (\ref{f00}), (\ref{YSp1}), (\ref{f1}), (\ref{fN}) and with (\ref{xijrp}), (\ref{thjrp}), (\ref{lajrp}), (\ref{Djrp}) yields:
\begin{align}\label{eqABD}
(A_0^{rr})^{T}B_0^{r}A_j^{r(r+p)}                          &=-D_j^{r(r+p)}, \qquad p\geq 1,\\
(A_0^{rr})^{T}B_0^{r}A_j^{rr}+(A_j^{rr})^{T}B_0^{r}A_0^{rr}&=C_j^{r}-D_j^{rr}, \qquad p=0.\nonumber
\end{align}
Moreover, from (\ref{simetrija}) it follows that $\Psi^{krs}_{n}$ for $n\geq 0$ and $r=s$ is symmetric, thus $\Xi(j,r,0)$, $\Theta(j,r,0)$, $\Lambda(j,r,0)$ (and hence $C_j^{r}-D_j^{rr}$) are symmetric, too.

To get $A_j^{r(r+p)}$ for $p\geq 1$ we solve a simple equation of the form $A^{T}X=B$ with given nonsingular $A$ and arbitrary $B$,
while to get $A_j^{rr}$ we solve the equation of the form $A^{T}X+X^{T}A=B$ with known nonsingular $A$ and symmetric $B$; 
the solution is $X=\frac{1}{2}(A^{T})^{-1}B+(A^{T})^{-1}Z$ with $Z$ skew-symmetric.
In particular,  
$A=(A_0^{rr})^{T}B_0^{r}$ with 
$(A^{T})^{-1}=((A_0^{rr})^{T}B_0^{r})^{-1}=
A_0^{r}(C_0^{r})^{-1}$ (see (\ref{GABA})).
This proves the algorithm in (\ref{EqT3}).

Furthermore, $\Xi(j,r,p)$, $\Theta(j,r,p)$, $\Lambda(j,r,p)$ (thus also $D_j^{r(r+p)}$ and $A_j^{r(r+p)}$) depend on the entries of $A_{\widetilde{j}}^{r's'}$ with (\ref{induA}).
%
It is straightforward to see that the algorithm in (\ref{EqT3}) allows to compute each entry from the entries that are already known. 
Moreover, the entries of $A_{j}^{rs}$ for either $r=s$, $\alpha_r\geq 2$, $j\in \{1,\ldots,\alpha_r-1\}$
or $s>r\geq 1$, $N \geq 2$
are determined uniquely by the entries of all $A_{\widetilde{j}}^{r's'}$ with 
$\widetilde{j}=0$, $s'=r'$ or $s'\leq r'$, $\widetilde{j}\in \{0,\ldots,\alpha_{r'}-1\}$ (chosen in (\ref{EqT2})), by the entries of all $Z_{\widetilde{j}}^{r'}$ with $\widetilde{j}\in \{1,\ldots,j-1\}$ 
(if $j\geq 2$) or $\widetilde{j}=j$, $r=r'$, and when $s>r$, $N \geq 2$ also by the entries of $Z_j^{r'}$ for all $r'$ (chosen in (\ref{EqT3})); $r',s'\in \{1,\ldots,N\}$.

%
%

If $B_0^{r},G_0^{r}$ are real, 
then by Sylvester's theorem the equation (\ref{GABA}) has a real solution $A_0^{rr}$ precisely when $B_0^{r},G_0^{r}$ are of the same inertia. The last statement of the lemma is then apparent.
\end{proof}

\begin{remark}
\begin{enumerate}[label=(\arabic*),ref={P\Roman*},wide=0pt,itemsep=2pt]
\item The equation in (\ref{EqT2}) is of the form $C=X^{T}BX$ with given nonsingular symmetric matrices $B$, $C$.
By Autonne-Takagi factorization (see e.g. \cite[Corolarry 4.4.4]{HornJohn}) 
$B=R^{T}I R$, $C=S^{T}IS$ for some nonsingular $R,S$ and the identity-matrix $I$. The above equation thus reduces to $I=Y^{T}Y$ with $Y=RXS^{-1}$. When 
$B$ and $C$ are real with the same inertia matrix $\widetilde{I}$, i.e. $B=R^{T}\widetilde{I} R$ and $C=S^{T}\widetilde{I} S$ for some real orthogonal $R$ and $S$, we get 
$\widetilde{I}=Y^{T}\widetilde{I} Y$ with $Y=RXS^{-1}$ (real pseudo-orthogonal).

\item One could consider the equation (\ref{eqFYFIY}) even when the diagonal blocks of $\mathcal{B}$, $\mathcal{C}$ are nonsingular. In this more general setting the equation $C=A^{T}BA$ is more involved, while the solution of the equation $A^{T}X+X^{T}A=B$ 
is known (see \cite{Brad}).
\end{enumerate}
\end{remark}

\begin{example}
We solve (\ref{eqFYFIY}) for $\mathcal{F}=E_4(I)\oplus E_2(I)\oplus I$, $\mathcal{B}=\mathcal{C}=\mathcal{I}:=I_4(I)\oplus I_2(I)\oplus I$. 
Set 
\small
\[
\mathcal{Y}=\begin{bmatrix}[cccc|cc|c]
A_1 & B_1 & C_1 & D_1  &  H_1  &  G_1  &  J_1\\
0   & A_1 & B_1   & C_1  &  0  &  H_1  &  0\\
0   & 0   & A_1   & B_1    &  0 & 0 & 0 \\
0   & 0   &  0  &  A_1   & 0 &0 & 0\\
\hline
0   & 0   & N_1 & P_1                         &  A_3  &  B_3  &  J_3\\
0   & 0   & 0   & N_1              &  0    &   A_3  &  0\\
\hline
0   & 0   & 0   & R_1                        &  0    &   R_3    & A_4 
\end{bmatrix}.
\]
\normalsize
We compute:
\small
\setlength{\arraycolsep}{3pt}
\begin{align*}
&\widetilde{\mathcal{Y}}\mathcal{Y}=\begin{bmatrix}[cccc|cc|c]
A_1^{T} & B_1^{T} & C_1^{T} & D_1^{T}  &  N_1^{T}  &  P_1^{T}  &  R_1^{T}\\
0   & A_1^{T} & B_1^{T}   & C_1^{T}  &  0  &   N_1^{T}  &  0\\
0   & 0   & A_1^{T}   & B_1^{T}    &  0 & 0 & 0 \\
0   & 0   &  0  &  A_1^{T}  & 0 &0 & 0\\
\hline
0   & 0   & H_1^{T} & G_1^{T}                        &  A_3^{T}  &  B_3^{T}  &  R_3^{T}\\
0   & 0   & 0   & H_1^{T}              &  0    &   A_3  &  0\\
\hline
0   & 0   & 0   & J_1^{T}                         &  0    &   J_3^{T}    & A_4 
\end{bmatrix}
\begin{bmatrix}[cccc|cc|c]
A_1 & B_1 & C_1 & D_1  &  H_1  &  G_1  &  J_1\\
0   & A_1 & B_1   & C_1  &  0  &   H_1  &  0\\
0   & 0   & A_1   & B_1    &  0 & 0 & 0 \\
0   & 0   &  0  &  A_1   & 0 &0 & 0\\
\hline
0   & 0   & N_1 & P_1                         &  A_3  &  B_3  &  J_3\\
0   & 0   & 0   & N_1              &  0    &   A_3  &  0\\
\hline
0   & 0   & 0   & R_1                        &  0    &   R_3    & A_4 
\end{bmatrix}=
\end{align*}
\scriptsize
\setlength{\arraycolsep}{3pt}
\begin{align*}
&=
\begin{bmatrix}[cccc|cc|c]
A_1^{T}A_1 & A_1^{T}B_1+B_1^{T}A_1 & A_1^{T}C_1+C_1^{T}A_1 & \mbox{\LARGE $*$} &  A_1^{T}H_1+N_1^{T}A_3  &  A_1^{T}G_1+B_1^{T}H_1+N_1^{T}B_3  &  N_1^{T}J_3+R_1^{T}A_4\\
           &                       & + B_1^{T}B_1+N_1^{T}N_1               &          &                         &           +P_1^{T}A_3+R_1^{T}R_3 & +A_1^{T}J_1 \\
           & A_1^{T}A_1 &  A_1^{T}B_1+B_1^{T}A_1   & A_1^{T}C_1+C_1^{T}A_1  &  0  &   A_1^{T}H_1+N_1^{T}A_3  &  0\\
   &                       &                          &       + B_1^{T}B_1+N_1^{T}N_1    &     &                          &   \\
   &                       & A_1^{T}A_1               &  A_1^{T}B_1+B_1^{T}A_1           &  0  & 0                      &      0 \\
   &    &    &  A_1^{T}A_1  &  0   &0 & 0\\
\hline 
   &    & &                      &  A_{3}^{T}A_3  &   A_3^{T}B_3+B_3^{T}A_3 + R_3^{T}R_3   &  A_3^{T}J_3+R_3^{T}A_4\\
   &    &    &              &         &   A_3^{T}A_3  &  0\\
\hline
   &    &    &                       &      &      & A_4 ^{T}A_4
\end{bmatrix}
\end{align*}
\normalsize
By comparing the diagonal of the diagonal blocks of the left-hand side and the right-hand side of $\widetilde{\mathcal{Y}}\mathcal{Y}=\mathcal{I}$ we deduce that $A_1,\ldots,A_4$ are any orthogonal matrices. Next, we choose $N_1$, $P_1$, $R_1$, $R_3$ arbitrarily. The diagonal blocks on the first upper diagonal yield equations $A_1^{T}H_1+N_1^{T}A_3=0$ and $A_3^{T}J_3+R_3^{T}A_4=0$, which further implies $H_1=-A_1N_1^{T}A_3$, $J_3=-A_3R_3^{T}A_4$; note that $(A_1^{T})^{-1}=A_1$, $(A_3^{T})^{-1}=A_3$. The last upper diagonal gives $N_1^{T}J_3+A_1^{T}J_1+R_1^{T}A_4=0$, thus $J_1=A_1(N_1^{T}A_3R_3^{T}A_4-R_1^{T}A_4)$.

By inspecting the first upper diagonal of the main diagonal blocks in $\widetilde{\mathcal{Y}}\mathcal{Y}=\mathcal{I}$ we obtain $A_1^{T}B_1+B_1^{T}A_1=0$ and $A_3^{T}B_3+B_3^{T}A_3+R_3^{T}R_3=0$, so we deduce $B_1,B_3$. Further, $A_1^{T}G_1+B_1^{T}H_1+N_1^{T}B_3 +P_1^{T}A_3+R_1^{T}R_3=0$ (observe the first upper diagonal of the first upper diagonal), so we get $G_1$.

The third and the fourth upper diagonal block of the first principal diagonal block give $A_1^{T}C_1+C_1^{T}A_1+B_1^{T}B_1+N_1^{T}N_1 =0$, $A_1^{T}D_1+B_1^{T}C_1+C_1^{T}B_1+D_1^{T}A_1+N_1^{T}P_1+P_1^{T}N_1+R_1^{T}R_1 =0$ (see $*$), therefore $C_1$, $D_1$ follow, respectively.
\end{example}

The solutions of the equation (\ref{eqFYFIY}) with a block diagonal matrix $\mathcal{C}=\mathcal{B}$ form a group with relatively simple generators. Recall that $\mathbb{U}$ is the set of matrices of the form (\ref{0T0}) with identity-matrices on the diagonals of the diagonal blocks. 

\begin{lemma}\label{lemauni}
The set $\mathbb{X}_{\mathcal{B}}$ of solutions of the equation (\ref{eqFYFIY}) for $\mathcal{C}=\mathcal{B}=\oplus_{r=1}^{N}(\oplus_{j=1}^{\alpha_r}B_{r})$ with $B_{r}\in GL_{m_r}(\mathbb{C})\cap S_{m_r}(\mathbb{C})$ is a semidirect product $\mathbb{X}_{\mathcal{B}}=\mathbb{O}_{\mathcal{B}}\ltimes \mathbb{V}_{\mathcal{B}}$, in which the group $\mathbb{O}_{\mathcal{B}}$ consists of all matrices of the form $\mathcal{Q}=\oplus_{r=1}^{N}(\oplus_{j=1}^{\alpha_r} Q_r)$ with $Q_r\in \mathbb{C}^{m_r\times m_r}$ such that $B_{r}=Q_r^{T}B_{r}Q_r$, and $\mathbb{V}_{\mathcal{B}}:=\mathbb{U}\cap \mathbb{X}_{\mathcal{B}}$ (hence unipotent of order at most $\alpha_1-1$ and in nilpotency class at most $\alpha_1$). Moreover, $\mathbb{V}_{\mathcal{B}}$ is generated by matrices of the form 
\vspace{-1mm}
\begin{align}\label{genV}
&\mathcal{V}=\bigoplus_{r=1}^{N}T(I_{m_r},V_1^{r},\ldots,V_{\alpha_r-1}^{r}),\\
&V_{1}^{r}:=\frac{1}{2}(B_{r})^{-1}Z_{1}^{r}, \qquad 
V_{n+1}^{r}:=\frac{1}{2}(B_{r})^{-1}\big(Z_{n+1}^{r}-\sum_{j=1}^{n}(V_j^{r})^{T}B_{r}V_{n-j+1}^{r}\big),\quad n\geq 1,\nonumber
\end{align}
\vspace{-1mm}
in which all $Z_{n}^{r}$ are skew-symmetric, and by matrices of the form
\vspace{-1mm}
\begin{align}\label{H2ptk}
&\mathcal{H}_{p,t}^{k}(F)=[(\mathcal{H}_{p,t}^{k}(F))_{rs}]_{r,s=1}^{N},\quad p<t, \qquad
(\mathcal{H}_{p,t}^{k}(F))_{rs}=
\left\{
\begin{array}{ll}
[0\quad \mathcal{U}_{rs}], & \alpha_r<\alpha_s\\
\begin{bmatrix}
\mathcal{U}_{rs}\\
0
\end{bmatrix}, & \alpha_r>\alpha_s\\
\mathcal{U}_{rs},& \alpha_r=\alpha_s
\end{array}\right.,
\end{align}
\vspace{-1mm}
where $F\in \mathbb{C}^{m_p\times m_t}$ and
\begin{align*}
&\mathcal{U}_{rs}=\left\{
\begin{array}{ll}
\oplus_{j=1}^{\alpha_r}I_{m_r}, &  r=s,\\
0,                        &  r\neq s 
\end{array}
\right., \quad \{r,s\}\subset\{p,t\},\qquad \nonumber
\mathcal{U}_{rr}=T(I_{m_r},A_1^{r},\ldots,A_{\alpha_r-1}^{rr}), \quad r\in \{p,t\},\\
&A_{j}^{pp}=\left\{\begin{array}{ll}
a_{n-1}B_{p}^{-1}(F^{T}B_tFB_p^{-1})^{n}B_p, & j=n(2k+\alpha-\beta)\\
0,                      & \textrm{otherwise}
\end{array}
\right., \qquad
a_{n}=-\frac{1}{2^{2n+1}}\frac{1}{n+1}{2n \choose n}\nonumber \\
&A_{j}^{tt}=\left\{\begin{array}{ll}
a_{n-1}B_t^{-1}(B_tFB_p^{-1}F^{T})^{n}B_t, & j=n(2k+\alpha-\beta)\\
0,                      & \textrm{otherwise}
\end{array}
\right., \nonumber \\
&\mathcal{U}_{pt}=
N_{\alpha_t}^{k}(F ),
\qquad
\mathcal{U}_{tp}=
N_{\alpha_t}^{k}(-B_p^{-1}F^{T}B_t).
\nonumber
\end{align*}
%
\end{lemma}

\begin{proof}
For any $\mathcal{X}_1,\mathcal{X}_2\in \mathbb{X}_{\mathcal{B}}$ we have: 
\vspace{-1mm}
\begin{align*}
\mathcal{F}(\mathcal{X}_1\mathcal{X}_2^{-1})^{T}\mathcal{F}\mathcal{B} (\mathcal{X}_1\mathcal{X}_2^{-1})
&=\mathcal{F}(\mathcal{X}_2^{-1})^{T}\mathcal{F}\mathcal{F}\mathcal{X}_1^{T}\mathcal{F}\mathcal{B} \mathcal{X}_1\mathcal{X}_2^{-1}
=\mathcal{F}(\mathcal{X}_2^{-1})^{T}\mathcal{F}\mathcal{B}\mathcal{X}_2^{-1}=\\
&=\mathcal{F}(\mathcal{X}_2^{-1})^{T}\mathcal{F}\mathcal{B}(\mathcal{B}^{-1}\mathcal{F}\mathcal{X}_2^{T}\mathcal{F}\mathcal{B})=\mathcal{B}.
\end{align*}
\vspace{-1mm}
Thus $\mathcal{X}_1\mathcal{X}_2^{-1}\in \mathbb{X}_{\mathcal{B}}$, so $\mathbb{X}_{\mathcal{B}}$ is a group.

We describe the structure of $\mathbb{X}_{\mathcal{B}}$. Lemma \ref{EqT} for $\mathcal{C}=\mathcal{B}=\oplus_{r=1}^{N}(\oplus_{j=1}^{\alpha_r}B_{0}^{r})$ implies that $\mathcal{X}\in \mathbb{X}_{\mathcal{B}}$ is of the form (\ref{0T0}) such that its diagonal blocks 
have $Q_0^{r}$ (satisfying $B_0^{r}=(Q_0^{r})^{T}B_0^{r}Q_0^{r}$) on the diagonal. Therefore $\mathcal{X}$
can be written as $\mathcal{X}=\mathcal{Q}\mathcal{Y}$ with $\mathcal{Q}\in \mathbb{O}_{\mathcal{B}}$ and $\mathcal{Y}\in \mathbb{U}$.
Clearly $\mathbb{O}_{\mathcal{B}}\subset \mathbb{X}_{\mathcal{B}}$ (hence $\mathcal{Y}\in \mathbb{X}_{\mathcal{B}}$), thus $\mathbb{X}_{\mathcal{B}}=\mathbb{O}_{\mathcal{B}}\ltimes \mathbb{V}_{\mathcal{B}}$, where $\mathbb{V}_{\mathcal{B}}=\mathbb{X}_{\mathcal{B}}\cap\mathbb{U}$.
Since $\mathbb{V}_{\mathcal{B}}$ is a subgroup of $\mathbb{U}$, it is a normal subgroup in $\mathbb{X}_{\mathcal{B}}$, unipotent of order at most $\alpha_1-1$, and nilpotent of class at most $\alpha_1$ (see Lemma \ref{lemanilpo}).

Next, we find matrices in $\mathbb{X}_{\mathcal{B}}$ that are of a simple form. First, set
%
\begin{align}\label{gen}
&\mathcal{D}_{\alpha,\beta}^{k}=\begin{bmatrix}
\Delta_{11} & \Delta_{12}\\
\Delta_{21} & \Delta_{22}
\end{bmatrix},\qquad \alpha>\beta, \quad 0\leq k\leq \beta-1 \nonumber\\
%
&\Delta_{11}=T(I_{m_1},A_1,\ldots,A_{\alpha-1}), \qquad 
\Delta_{22}=T(I_{m_2},D_1,\ldots,D_{\beta-1}), \qquad \\
&\Delta_{21}=\begin{bmatrix}
0 & N_{\beta}^{k}(F)\\
\end{bmatrix},\qquad
\Delta_{12}=\begin{bmatrix}
T(G_0,G_1,\ldots,G_{\beta-1})\\
0
\end{bmatrix},\nonumber
\end{align}
%
in which $N_{\beta}^{k}(F)$ is a $\beta\times \beta$ block matrix with $F\in \mathbb{C}^{m_1\times m_2}$ on the $k$-th diagonal above the main diagonal and zeros othervise, $A_j\in \mathbb{C}^{m_1\times m_1}$, $D_j\in \mathbb{C}^{m_2\times m_2}$, and $G_j\in \mathbb{C}^{m_2\times m_1}$ for all $j$.
Further, suppose $\mathcal{D}_{\alpha,\beta}^{k}$ is a solution of  
the matrix equation 
\[
\mathcal{B}_{\alpha,\beta}=\mathcal{F}_{\alpha,\beta}(\mathcal{D}_{\alpha,\beta}^{k})^{T}\mathcal{F}_{\alpha,\beta}\mathcal{B}_{\alpha,\beta}\mathcal{D}_{\alpha,\beta}^{k}, \qquad \mathcal{B}_{\alpha,\beta}=I_{\alpha}(B)\oplus I_{\beta}(C),\quad \mathcal{F}_{\alpha,\beta}=E_{\alpha}(I_{m_1})\oplus E_{\beta}(I_{m_2}),
\]
where $B\in  \mathbb{C}^{m_1\times m_1}$ and $C\in  \mathbb{C}^{m_2\times m_2}$.
Blockwise we have
\small
\begin{align}\label{eqN1}
&I_{\alpha}(B)=T(I_{m_1},A_1^{T},\ldots,A_{\alpha-1}^{T})I_{\alpha}(B)T(I_{m_1},A_1,\ldots,A_{\alpha-1})+
\begin{bsmallmatrix}
N_{\beta}^{k}(F^T)\\
0
\end{bsmallmatrix}
I_{\beta}(C)
\begin{bsmallmatrix}
0 & N_{\beta}^{k}(F)\\
\end{bsmallmatrix},\\
&
\label{eqN2}
0=T(I_{m_1},A_1^{T},\ldots,A_{\alpha-1}^{T})I_{\alpha}(B)\begin{bsmallmatrix}
T(G_0,G_1,\ldots,G_{\beta-1})\\
0
\end{bsmallmatrix}+
\begin{bsmallmatrix}
N_{\beta}^{k}(F^{T})\\
0
\end{bsmallmatrix}
I_{\beta}(C)T(I_{m_2},D_1,\ldots,D_{\beta-1}),\\
&
\label{eqN3}
I_{\beta}(C)=T(I_{m_2},D_1^{T},\ldots,D_{\beta-1}^{T})I_{\beta}(C)T(I_{m_2},D_1,\ldots,D_{\beta-1})\\
&\qquad \quad+\begin{bsmallmatrix}
0 & T(G_0^{T},G_1^{T},\ldots,G_{\beta-1}^{T})
\end{bsmallmatrix}
I_{\alpha}(B)
\begin{bsmallmatrix}
T(G_0,G_1,\ldots,G_{\beta-1})\\
0
\end{bsmallmatrix}.\nonumber
\end{align}
\normalsize 
To determine $\mathcal{D}_{\alpha,\beta}^{k}$ we follow the algorithm in Lemma \ref{EqT}.

We first simplify the notation by defining $A_0:=I_{m_1}$ and $D:=I_{m_2}$. By comparing the first row of the left-hand and the right-hand side of (\ref{eqN1}) we get 
\begin{align*}
0=\sum_{j=0}^{n} A_j^{T}BA_{n-j}, \quad 
1\leq n\neq 2k+\alpha-\beta,\qquad 
0=\sum_{j=0}^{2k+\alpha-\beta} A_j^{T}BA_{2k+\alpha-\beta-j}+F^{T}CF, 
\end{align*}
If $\alpha-\beta+k\geq 2$ we satisfy the first equation for $1\leq n\leq 2k+\alpha-\beta-1$ by choosing $A_1=\ldots=A_{2k+\alpha-\beta-1}=0$. The second equation then yields $-F^{T}CF=A_{2k+\alpha-\beta}^{T}B+BA_{2k+\alpha-\beta}$ (in the case $\alpha-\beta+k=1$ as well) and we take $A_{2k+\alpha-\beta}=-\frac{1}{2}B^{-1}F^{T}CF$. 
The first equation for $2k+\alpha-\beta+1\leq n\leq 2(2k+\alpha-\beta)$ further reduces to: 
\begin{align*}\label{eqr2}
&0=A_{j}^{T}B+BA_{j}^{T},\qquad\quad  2k+\alpha-\beta+1\leq j\leq 2(2k+\alpha-\beta) -1 \quad (\textrm{if }\alpha-\beta+k\geq 2 ), \nonumber\\
&0=A_{2(2k+\alpha-\beta)}^{T}B+A_{2k+\alpha-\beta}^{T}B A_{2k+\alpha-\beta}+BA_{2(2k+\alpha-\beta)}.
\end{align*}
Hence we can choose $A_{j}=0$ for $2k+\alpha-\beta+1\leq j\leq 2(2k+\alpha-\beta) -1$ (if $\alpha-\beta+k\geq 2$) and $A_{2(2k+\alpha-\beta)}=-\frac{1}{8}(B^{-1}F^{T}CF)^{2}$. By continuing in this manner we obtain:
\begin{equation}\label{eqA1}
A_{j}=\left\{\begin{array}{ll}
a_{n-1}(B^{-1}F^{T}CF)^{n}, & j=n(2k+\alpha-\beta),\; n\in \mathbb{N}\\
0,                      & \textrm{otherwise}
\end{array}
\right.,
\end{equation}
where $a_0=-\frac{1}{2}$ and $a_{n}=-\frac{1}{2}\sum_{j=0}^{n-1}a_ja_{n-j-1}$ for $n\in \mathbb{N}$. 
The generating function associated with the sequence $a_n$ is $f(t):=\sum_{j=0}^{\infty}a_jt^{j}$. Observe that $f(t)=-\frac{1}{2}t(f(t))^{2}-\frac{1}{2}$, thus $f(t)=-\frac{1}{t}\big(1+(1-t)^{\frac{1}{2}}\big)$ and 
we obtain $a_{n}=-\frac{1}{2^{2n+1}}\frac{1}{n+1}{2n \choose n}$.
For the basic theory of generating functions see e.g. \cite[Chapter 2]{Riordan}).

%
%
We now compare the 
entries in the first row of the left-hand and the right-hand side of (\ref{eqN2}) and get the following equations:
\begin{align}\label{eqr3}
0=&\sum_{j=0}^{n}A_{j}^{T}BG_{n-j},\qquad  0\leq n\leq k-1,\quad (\textrm{if }k\geq 1 )\nonumber \\ 
0=&\sum_{j=0}^{k}A_{j}^{T}BG_{k-j}+F^{T}C, \quad  \\
0=&\sum_{j=0}^{n}A_{j}^{T}BG_{n-j}+F^{T}CD_{n-k}, \quad n\geq k+1,
\nonumber
\end{align}
The first two equations immediately imply
\begin{equation}\label{eqG1}
G_0=\ldots=G_{k-1}=0\quad (\textrm{if }k\geq 1), \qquad G_{k}=-B^{-1}F^{T}C.
\end{equation}

By comparing the 
entries in the first row of the left-hand and the right-hand side of (\ref{eqN3}), we obtain:
\begin{align}\label{DCD}
0=\sum_{j=0}^{n}D_{j}^{T}CD_{n-j}+\sum_{j=\alpha-\beta}^{n}G_{j-(\alpha-\beta)}^{T}BG_{n-1}, \qquad n\geq 1.
\end{align}
Using (\ref{eqG1}) we deduce that the second summand  on the right-hand side of (\ref{DCD}) vanishes for $1\leq n\leq \alpha-\beta+2k-1$, $\alpha-\beta+k\geq 2$, thus
\begin{align*}
&0=\sum_{j=0}^{n}D_{j}^{T}CD_{n-j},\quad 1\leq n\leq \alpha-\beta+2k-1 \quad (\textrm{if }\alpha-\beta+k\geq 2 ), \nonumber \\
&0=\sum_{j=0}^{2k+\alpha-\beta}D_{j}^{T}CD_{2k+\alpha-\beta-j}+G_k^{T}BG_k.\nonumber
\end{align*}
Therefore we choose 
\begin{align}\label{eqD1}
&D_1=\ldots=D_{2k+\alpha-\beta-1}=0\quad (\textrm{if }\alpha-\beta+k\geq 2 ),\\
&D_{2k+\alpha-\beta}=-\frac{1}{2}C^{-1}G_k^{T}BG_k=-\frac{1}{2}FB^{-1}F^{T}C.\nonumber
\end{align}

Using (\ref{eqA1}) and (\ref{eqD1}) for $\alpha-\beta+k\geq 2$,
the last equation of (\ref{eqr3}) reduces to $0=BG_{n}$ for $\alpha-\beta+2k-1\geq n\geq k+1$, hence
\begin{equation}\label{eqG2}
G_{k+1}=\ldots=G_{2k+\alpha-\beta-1}=0\quad (\textrm{if }\alpha-\beta+k\geq 2 ).
\end{equation}
Further, we apply (\ref{eqA1}), (\ref{eqG1}), (\ref{eqD1}), (\ref{eqG2}) to the last equation of (\ref{eqr3}) for $n=\alpha-\beta+2k$. 
If $k\geq 1$ 
we obtain $BG_{2k+\alpha-\beta}=0$, while for $\alpha-\beta\geq 2$, $k=0$ we get
\[
0=BG_{\alpha-\beta}+A_{\alpha-\beta}^{T}BG_0+F^{T}CD_{\alpha-\beta}=BG_{\alpha-\beta}-\tfrac{1}{2}F^{T}CFG_0-\tfrac{1}{2}F^{T}G_0^{T}BG_0=BG_{\alpha-\beta}.
\]
Similarly for $k=0$, $\alpha-\beta=1$ we deduce $BG_1=0$. In any case we have
\begin{equation}\label{G2k}
G_{2k+\alpha-\beta}=0.
\end{equation}

If $\alpha-\beta+k\geq 2$ we use (\ref{eqG1}), (\ref{eqD1}), (\ref{eqG2}), (\ref{G2k}) to see that the second summand on the right-hand side of (\ref{DCD}) for $\alpha-\beta+2k+1\leq n\leq 2(\alpha-\beta)+3k$ vanishes, while the first summand is equal to $D_{n}^{T}C+CD_{n}^{T}$,
thus:
\[
0=D_{n}^{T}C+
CD_{n}^{T}, \qquad \alpha-\beta+2k+1\leq n\leq 2(\alpha-\beta)+3k. \\
\]
We take 
\begin{equation}\label{eqD2}
D_n=0,\qquad \alpha-\beta+2k+1\leq n\leq 2(\alpha-\beta)+3k \qquad (\textrm{if }\alpha-\beta+k\geq 2).
\end{equation}
Using (\ref{eqA1}), (\ref{eqD1}), (\ref{eqD2}), 
the third equation of (\ref{eqr3}) for $\alpha-\beta+2k+1\leq n \leq 2(\alpha-\beta)+3k$ 
reduces to $BG_n=0$; it is clear for $n\neq \alpha-\beta+3k$, while for $n= \alpha-\beta+3k$:
%
\begin{align}\label{eqBGAD}
&0=BG_{\alpha-\beta+3k}+A_{\alpha-\beta+2k}^{T}BG_{k}+F^{T}CD_{\alpha-\beta+2k}\\
&0=BG_{\alpha-\beta+3k}+\tfrac{1}{2}F^{T}CFB^{-1}F^{T}C-\tfrac{1}{2}F^{T}CFB^{-1}F^{T}C=BG_{\alpha-\beta+3k}.\nonumber
\end{align}
It yields:
\begin{equation}\label{eqG3}
G_n=0, \qquad  \alpha-\beta+2k+1\leq n \leq 2(\alpha-\beta)+3k. 
\end{equation}
Equations (\ref{DCD}) for (\ref{eqG1}),(\ref{eqD1}),(\ref{eqG2}),(\ref{eqG3}),(\ref{eqD2}) then give
\begin{align*}
&0=CD_j+D_j^{T}C, \qquad 2(\alpha-\beta)+3k+1\leq j \leq 2(\alpha-\beta+2k)-1, \qquad \\
&0=CD_{2(\alpha-\beta+2k)}+D_{\alpha-\beta+2k}^{T}CD_{\alpha-\beta+2k}+D_{2(\alpha-\beta+2k)}^{T}C.
\end{align*}
We take 
\begin{align}\label{Dabk}
&D_n=0,\qquad 2(\alpha-\beta)+3k+1\leq n \leq 2(\alpha-\beta+2k)-1,\\
&D_{2(\alpha-\beta+2k)}=-\frac{1}{2}D_{\alpha-\beta+2k}^{T}CD_{\alpha-\beta+2k}=-\frac{1}{8}(C^{-1}G_k^{T}BG_k)^{2}.\nonumber
\end{align}
From (\ref{eqr3}) for (\ref{eqA1}),(\ref{eqD1}),(\ref{eqD2}) we further deduce 
\begin{equation}\label{Gabk}
G_{n}=0, \qquad 2(\alpha-\beta)+3k+1 \leq n\leq 2(\alpha-\beta+2k) .
\end{equation}
If $\alpha-\beta=1$, $k=0$ then (\ref{DCD}) yields $0=CD_2+(D_1)^{T}CD_1+D_2^{T}C$ and we choose $D_{2}=-\frac{1}{8}(C^{-1}G_0^{T}BG_0)^{2}$. Further, similarly as in (\ref{eqBGAD}) we apply (\ref{eqr3}) to get $G_2=0$ from $0=BG_{2}+A_{1}^{T}BG_{0}+F^{T}CD_{1}$. Thus (\ref{Dabk}), (\ref{Gabk}) are valid in this case as well.

By continuing this proces we eventually obtain:
\begin{align}\label{eqAA1}
&G_j=\left\{\begin{array}{ll}
B^{-1}F^{T}C, & j=k\\
0,                      & \textrm{otherwise}
\end{array}
\right.,\\
&D_{j}=\left\{\begin{array}{ll}
a_{n-1}(FB^{-1}F^{T}C)^{n}, & j=n(2k+\alpha-\beta)\\
0,                      & \textrm{otherwise}
\end{array}
\right.,\qquad
a_{n}=-\frac{1}{2^{2n+1}}\frac{1}{n+1}\small{{2n \choose n}}.\nonumber
\end{align}
%
%

Next, we compute $(\mathcal{D}_{\alpha,\beta}^{k}(F))^{-1}=\mathcal{B}_{\alpha,\beta}^{-1}\mathcal{F}_{\alpha,\beta}(\mathcal{D}_{\alpha,\beta}^{k}(F))\mathcal{F}_{\alpha,\beta}\mathcal{B}_{\alpha,\beta}
=\begin{bmatrix}
\Delta_{11}' & \Delta_{12}'\\
\Delta_{21}' & \Delta_{22}'
\end{bmatrix}$ with
\begin{align*}
&\Delta_{11}'=T(I_m,A_1',\ldots,A_{\alpha-1}'), \quad 
A_{j}'=\left\{\begin{array}{ll}
a_{n-1}B^{-1}(F^{T}CFB^{-1})^{n}B, & j=n(2k+\alpha-\beta)\\
0,                      & \textrm{otherwise}
\end{array}
\right.,  \\
&\Delta_{22}'=T(I_n,D_1',\ldots,D_{\beta-1}'),\quad
D_{j}'=\left\{\begin{array}{ll}
a_{n-1}C^{-1}(CFB^{-1}F^{T})^{n}C, & j=n(2k+\alpha-\beta)\\
0,                      & \textrm{otherwise}
\end{array}
\right., \\
&\Delta_{21}'=\begin{bmatrix}
0 & N_{\beta}^{k}(F)\\
\end{bmatrix},\qquad
\Delta_{12}'=\begin{bmatrix}
N_{\beta}^{k}(-B^{-1}F^{T}C)\\
0
\end{bmatrix},
\qquad
a_{n}=-\frac{1}{2^{2n+1}}\frac{1}{n+1}{2n \choose n}.
\end{align*}
%
%
%

Set $\mathcal{K}_{p,t}^{k}(F)\in\mathbb{V}_{\mathcal{B}}$ to be an $N\times N$ block matrix such that its principal submatrix formed by blocks in the $p$-th and the $t$-th columns and rows is equal to $\mathcal{D}_{\alpha_p,\alpha_t}^{k}(F)$, while the submatrix formed by all other blocks is the identity matrix. 
Clearly $\mathcal{H}_{p,t}^{k}(F):=(\mathcal{K}_{p,t}^{k}(F))^{-1}$ is of the same form as $\mathcal{K}_{p,t}^{k}(F)$, only with $(\mathcal{D}_{\alpha_p,\alpha_t}^{k}(F))^{-1}$ as a principal submatrix formed by blocks in the $p$-th and the $t$-th columns and rows.
%

We use the inductive procedure of multiplying $\mathcal{Y}\in \mathbb{V}_{\mathcal{B}}$ by matrices of the form $\mathcal{K}_{p,t}^{k}(F)$ for the appropriate $p,t,k,F$. To describe the inductive step, suppose that
during the process we have a matrix that by a slight abuse of notation is still called $\mathcal{Y}$, and such that the blocks under the main diagonal in the first $p-1$ columns vanish (i.e. $\mathcal{Y}_{rs}$ vanishes for $p,r>s$), and the first $\alpha_p-\alpha_{p+1}+k$ columns of $\mathcal{Y}_{rp}$ for $r>p$ vanish.
Let $t$ be the largest index such that $(\mathcal{Y}_{tp})_{1(\alpha_p-\alpha_{p+1}+k+1)}\neq 0$, i.e 
$
(\mathcal{Y}_{tp})_{(1)}=
\begin{bsmallmatrix}
0 & \ldots & 0 & R_{k-\alpha_{p+1}+\alpha_t}^{tp} & \ldots & R_{\alpha_{t}-1}^{tp}
\end{bsmallmatrix}$ with all $R_{j}^{tp}\in \mathbb{C}^{m_t\times m_p}$ and $R_{k-\alpha_{p+1}+\alpha_t}^{tp}\neq 0$. 
We multiply $\mathcal{Y}$ with $\mathcal{K}_{p,t}^{k}(-R_{k-\alpha_{p+1}+\alpha_t}^{tp})$ to get $\mathcal{Y}'$ of the same form as $\mathcal{Y}$, and with $(\mathcal{Y}_{tp}')^{(k+\alpha_t-\alpha_{p+1}+1)}= 0$. It is apparent for $\mathcal{Y}_{rs}'$ with $r, p >s$ or $t>r>s=p$, while for $r\geq t$, $s=p$
we have
\vspace{-1mm}
\begin{align*}
(\mathcal{Y}_{tp}')_{(1)}
=&
\begin{bmatrix}
0 & \ldots & 0 & R_{k-\alpha_{p+1}+\alpha_t}^{tp} & \ldots & R_{\alpha_{t}-1}^{tp}
\end{bmatrix}
T(I_{m_p},A_{1}^{pp},A_{1}^{pp},\ldots,A_{\alpha_p-1}^{pp})\\
&+
T(I_{m_t},A_{1}^{tt},\ldots,A_{\alpha_t-1}^{tt})
\begin{bmatrix}0 &  N_{\alpha_p-\alpha_t}^{k+\alpha_t-\alpha_{p+1}}(-R_{k-(\alpha_{p+1}-\alpha_t)}^{tp})  \end{bmatrix}\\
=& \begin{bmatrix}
0 & \ldots & 0 & S_{k-\alpha_{p+1}+\alpha_t+1}^{tp} & \ldots & R_{\alpha_{t}-1}^{tp}
\end{bmatrix},\\
%
%
%
(\mathcal{Y}_{rp}')_{(1)}
=&
\begin{bmatrix}
0 & \ldots & 0 & R_{k-\alpha_{p+1}+\alpha_t}^{rp} & \ldots & R_{\alpha_{t}-1}^{rp}
\end{bmatrix}
T(I_{m_p},\ldots,A_{\alpha_p-1}^{pp})\\
&+
\begin{bmatrix}
0 & * & \ldots & *
\end{bmatrix}
\begin{bmatrix}
0 & N_{\alpha_p-\alpha_t}^{k+\alpha_t-\alpha_{p+1}}(-R_{k-(\alpha_{p+1}-\alpha_t)}^{tp})  \end{bmatrix}\\
=&  \begin{bmatrix}
0 & \ldots & 0 & S_{k-\alpha_{p+1}+\alpha_t+1}^{rp} & \ldots & R_{\alpha_{t}-1}^{rp}
\end{bmatrix}, \qquad  r>t,
\end{align*}
%
%
%
for some $S_{j}^{sp}\in \mathbb{C}^{m_s\times m_p}$ with $s\in \{r,t\}$.
This process (i.e. choosing the appropriate $\{p_j,t_j,k_j,F_j\}_{j=1}^{n}$) eventually yields a block upper-triangular matrix and it is of the form (\ref{0T0}) such that the blocks on the main diagonal are block upper-triangular Toeplitz with identities on the diagonals; we denote it by $\mathcal{V}$:
\vspace{-2mm}
\[
\mathcal{X}=\mathcal{Q}\mathcal{Y}=\mathcal{Q}\mathcal{V}\prod_{j=1}^{n}(\mathcal{K}_{p_j,t_j}^{k_j}(F_j))^{-1}=\mathcal{Q}\mathcal{V}\prod_{j=1}^{n}\mathcal{H}_{p_j,t_j}^{k_j}(F_j).
\]
%
The inverse of a nonsingular block upper-triangular Toeplitz matrix is again a block upper-triangular Toeplitz, hence $\mathcal{V}^{-1}$ is block upper-triangular. On the other hand $\mathcal{V}$ is a solution of the equation (\ref{eqFYFIY}), so $\mathcal{V}^{-1}=\mathcal{B}^{-1}\mathcal{F}\mathcal{V}\mathcal{F}\mathcal{B}$ is also a block lower-triangular matrix. 
Hence $\mathcal{V}=\oplus_{r=1}^{N}T(I_{m_r},V_1^{r},\ldots,V_{\alpha_r-1}^{r})$;
%
%
the algorithm that provides the solution of (\ref{eqFYFIY}) (see Lemma \ref{EqT}) yields equations that give (\ref{genV}):
%
\begin{align*}
&(B_0^{r})^{T}V^{r}_{1}+(V_1^{r})^{T}B_0^{r}=0,\\
&(B_0^{r})^{T}V^{r}_{n+1}+(V_{n+1}^{r})^{T}B_0^{r}=-\sum_{j=1}^{n}(V_j^{r})^{T}B_0^{r}V_{n+1-j}^{r}, \quad n\geq 1.
\end{align*}
\vspace{-2mm}
This concludes the proof of the lemma. 
\end{proof}

\section{Proof of Theorem \ref{stabs}}\label{PfTh}

We begin with a direct simple proof of Corollary \ref{orbdim},
since the tangent space of $\Orb(A)$ at $A$ (see $T_A$ in (\ref{tans})) is easily computed. Indeed,
%
%
if $Q(t)$ is a complex-differentiable path of orthogonal matrices with $Q(0)=I$, then
\[
\frac{d}{dt}\Big|_{t=0}\big((Q(t))^{T}AQ(t)\big)=(Q'(0))^{T}A+AQ'(0),
\]
and differentiation of $(Q(t))^{T}Q(t)=I$ at $t=0$ yields $(Q'(0))^{T}+Q'(0)=0$; conversely, for any $X=-X^{T}$ we have $e^{0\cdot X}=I$ and $\frac{d}{dt}\big|_{t=0}(e^{tX})=X$.

Observe that the dimension of $T_A$ in (\ref{tans}) is precisely the codimension of the solution space of $X^{T}A+AX=0$ with $X=-X^{T}$ (with respect to the space of all skew-symmetric matrices).
If $J$ is the Jordan form of $A$ with $A=P^{-1}JP$, we get
\[
JY=YJ, \qquad Y=PXP^{-1}, \quad X=-X^{T}.
\]
Thus $Y=[Y_{jk}]_{j,k=1}^{N}$ has rectangle upper-triangular Toeplitz blocks $Y_{jk}$ (see Theorem \ref{posls}), 
and
$Y=-P^{2}Y^{T}P^{-2}$. 
Note that (see e.g. \cite[Theorem 4.4.24]{HornJohn}):
\vspace{-1mm}
\begin{equation}\label{P}
K_{\alpha}(\lambda)=P_{\alpha}J_{\alpha}(\lambda)P_{\alpha}^{-1},\quad P_{\alpha}:=\tfrac{1}{\sqrt{2}}(I_{\alpha}+iE_{\alpha}),\qquad E_{\alpha}=
\begin{bsmallmatrix}
 0                 &      & 1\\
            &   \iddots     &  \\
1            &           &  0\\
\end{bsmallmatrix} \quad (\alpha\textrm{-by-}\alpha),
\end{equation}
%
in which 
$E_{\alpha}$
is the backward identity-matrix (with ones on the anti-diagonal); and $P_{\alpha}^{2}=E_{\alpha}$. If $A$ is of the form (\ref{NF1s}), then $Y=-EY^{T}E$,
in which $E$ is a direct sum of backward identity matrices and it is partitioned conformally to $Y$. 
In view of (\ref{ETE}), (\ref{AkrT}) we further obtain that all $Y_{jj}=0$ and 
$Y_{jk}=
\begin{bsmallmatrix}
T_{jk}\\
0
\end{bsmallmatrix}$, 
$Y_{kj}=
\begin{bsmallmatrix}
0 & T_{kj}
\end{bsmallmatrix}$
are related with $T_{jk}=-T_{kj}$ (both $T_{jk}$, $T_{kj}$ are upper-triangular Toeplitz). Corollary \ref{orbdim} now follows.



We now prove Theorem \ref{stabs}.


\begin{proof}[Proof of Theorem \ref{stabs}]
Given a symmetric matrix $S$ we need to solve the equation
provided as part of the proof of
\begin{equation}\label{SQQS}
SQ= QS,
\end{equation}
where $Q$ is an orthogonal matrix and
provided as part of the proof of
\[
S=
\bigoplus_{r=1}^{N}\Big( \bigoplus_{j=1}^{m_r} K_{\alpha_r}(\lambda)\Big),\qquad \lambda \in \mathbb{C}.
\]
We shall first use Theorem \ref{posls} to solve (\ref{SQQS}) on $Q$. Taking into account that $Q$ satisfies $Q^{T}Q=I$ ($I$ is the identity matrix), it will yield a certain matrix equation and further restricting the form of $Q$; at this point Lemma \ref{EqT} will be applied. 

We have 
\vspace{-1mm}
\begin{equation*}
S=P^{-1}JP,
\qquad
J=\bigoplus_{r=1}^{N}\Big( \bigoplus_{j=1}^{m_r} J_{\alpha_r}(\lambda)\Big),\qquad P=\bigoplus_{r=1}^{N}\Big( \bigoplus_{j=1}^{m_r} P_{\alpha_j}\Big),
\end{equation*}
where $P_{\alpha_j}$ is defined in (\ref{P}).
The equation (\ref{SQQS}) thus transforms to 
\vspace{-1mm}
\[
JX=XJ, \qquad X=PQP^{-1}.
\]

\vspace{-1mm}
From Theorem \ref{posls} (\ref{posls2}) we obtain 
%
%
where $X=[X_{rs}]_{r,s=1}^{N}$ with further $X_{rs}$ is a $m_r\times m_s$ block matrix whoose blocks of dimension $\alpha_r\times \alpha_s$ are of the form 
%
%
%
\vspace{-1mm}
\begin{equation}
\left\{
\begin{array}{ll}
[0\quad T], & \alpha_{r}<\alpha_{s}\\
\begin{bmatrix}
T\\
0
\end{bmatrix}, & \alpha_{r}>\alpha_{s}\\
T,& \alpha_r=\alpha_{s}
\end{array}\right.,
\end{equation}
where $T\in \mathbb{C}^{m\times m}$, $m=\min\{\alpha_r, \alpha_s\}$ is a complex upper-triangular Toeplitz matrix.

Since $P_{\alpha}=P_{\alpha}^T$, $P_{\alpha}^{-1}=\overline{P}_{\alpha}$, 
$P_{\alpha}^{2}=-\overline{P'}_{\alpha}^2=-P_{\alpha}^{-2}=iE_{\alpha}$, we deduce $P^{2}=\overline{P}^2=-P^{-2}=iE$, where $E:=\oplus_{r=1}^{N}\left(\oplus_{j=1}^{m_r} (E_{\alpha_r}) \right)$. Therefore $I=Q^TQ$ if and only if 
\begin{align}\label{QTQS}
I=   & (P^TX^T(P^{-1})^T)(P^{-1}X P )\nonumber\\
 I  = & P(P^TX^T(P^{-1})^T)(P^{-1}X P )P^{-1},\nonumber\\
 I = & P^2X^TP^{-2}X \\
  I = & iE X^{T} (-iE)X\nonumber\\
    I = & E X^{T} E X.\nonumber
\end{align}

\vspace{-2mm}

Proceed by conjugating with the permutation matrix $\Omega=\oplus_{r=1}^{N}\Omega_{\alpha_r,m_r} $ as in (\ref{otxo}):
%
\begin{align}\label{ortoC3S}
 &I = (\Omega^{T}E\Omega)(\Omega^{T} X^{T} \Omega)(\Omega^{T}E\Omega)(\Omega^{T}X\Omega)\\
 &I = \mathcal{F}\mathcal{X}^{T}\mathcal{F}\mathcal{X},\nonumber
\end{align}
%
where $\mathcal{F}=\Omega^{T}E\Omega=\oplus_{j=1}^{N}E_{\alpha_r}(E_{m_r})$ and $\mathcal{X}=\Omega^{T}X\Omega$ (see (\ref{otxo})) are 
of the form (\ref{0T0}) with block rectangular upper-triangular Toeplitz blocks.
%
%
By applying Lemma \ref{EqT} for $\mathcal{B}=\mathcal{C}=I$ and Lemma \ref{lemauni} with $\mathcal{B}=I$ to (\ref{ortoC3S}) we 
conclude the proof.
\end{proof}

\vspace{-2mm}

\begin{remark}
\begin{enumerate}[label=(\arabic*),ref={P\Roman*},wide=0pt,itemsep=2pt]
\item The equation (\ref{QTQS}) is very similar to the equation that we obtained in \cite[Proof of Theorem 1.1]{TSH} when examining orthogonal *congruence of certain Hermitian matrices. However, to compute the isotropy groups under orthogonal *congruence, a more detailed analysis of a few more cases would need to be done (due to the existence of three different types of normal forms). 
 
\item Canonical forms under orthogonal similarity are known for skew-symmet\-ric and orthogonal matrices, too. Using the same general approach as in the case of symmetric matrices,  
isotropy groups 
are described by matrix equations which involve an important difference in comparison to equations that we deal in this paper (Lemma \ref{EqT} and Lemma \ref{lemauni}). We expect that by developing some special techniques, similar results can be obtained. 
\end{enumerate}
\end{remark}

\vspace{-1mm}

\textit{Acknowledgement.}
This research was supported by Slovenian Research Agency (grant no. P1-0291).

\vspace{-3mm}

\end{document}